\documentclass[11pt,a4paper]{article}
\usepackage{a4wide,amsfonts,amsmath,latexsym,amssymb,euscript,eufrak,graphicx,units,mathrsfs}

\usepackage{stmaryrd}
\usepackage{color}
\usepackage{float}
\newfloat{figure}{H}{lof}
\floatname{figure}{\figurename}

\DeclareMathAlphabet{\eufrak}{U}{}{}{}  
\SetMathAlphabet\eufrak{normal}{U}{euf}{m}{n}
\SetMathAlphabet\eufrak{bold}{U}{euf}{b}{n}

\usepackage{amsmath, amsthm, amsfonts, amssymb}

\numberwithin{equation}{section}

\newcommand{\diff}[1]{\operatorname{d}\ifthenelse{\equal{#1}{}}{\,}{#1}}
\def\real{{\mathord{{\rm I\kern-2.8pt R}}}}        
\def\inte{{\mathord{{\rm I\kern-2.8pt N}}}}
\def\PP{{\mathord{{\rm I\kern-2.8pt P}}}}

\def\real{{\mathord{\mathbb R}}}

\def\inte{{\mathord{\mathbb N}}}

\def\R{\right}

\def\HH{\EuFrak H}

\newcommand{\e}{\varepsilon}

\def\R{\mathbb{R}}

\def\E{\mathbb{E}}
\def\Ex{\mathbb{E}}

\def\real{\mathbb{R}}
\newcommand{\abs}[1]{\left\vert #1 \right\rvert}
\newcommand{\norm}[1]{\left\lVert #1 \right\rVert}

\newtheorem{prop}{Proposition}[section]
\newtheorem{proposition}{Proposition}[section]

\newtheorem{lemma}[prop]{Lemma}

\newtheorem{theorem}[prop]{Theorem}

\allowdisplaybreaks

\begin{document}

\begin{center}
{\large \textbf{Continuous Breuer-Major theorem: tightness and non-stationarity}}\\[0pt]
~\\[0pt]
Simon Campese, Ivan Nourdin
\footnote{%
Universit\'e du Luxembourg, Maison du Nombre, 6 avenue de la Fonte, 
L-4364 Esch-sur-Alzette, Grand Duchy of Luxembourg.
{\tt \{simon.campese,ivan.nourdin\}@uni.lu}}
and David Nualart  
\footnote{%
Department of Mathematics, University of Kansas, Lawrence, KS 66045, USA.
{\tt nualart@ku.edu}  }
 \footnote{David Nualart was supported by the NSF grant  DMS 1512891}
\\[0pt]
{\it Universit\'e du Luxembourg and University of Kansas}\\
~\\[0pt]
\end{center}

{\small \noindent \textbf{Abstract:} 
Let $Y=(Y(t))_{t\geq0}$ be a zero-mean Gaussian stationary process with covariance function 
$\rho:\R\to\R$ satisfying $\rho(0)=1$. 
Let $f:\R\to\R$ be a square-integrable function with respect to the standard Gaussian measure, and suppose the Hermite rank of $f$ is $d\geq 1$. If $\int_\R |\rho(s)|^dds<\infty$, then the celebrated 
Breuer-Major theorem (in its continuous version) asserts that
the finite-dimensional distributions of
$Z_\e:=\sqrt{\e}\int_0^{\cdot/\e}f(Y(s))ds$ converge to those of $\sigma W$ as $\e\to 0$, where $W$ is a standard Brownian motion and $\sigma$ is some explicit constant.
Since its first appearance in 1983, this theorem
has become a crucial probabilistic tool in different areas, for instance in signal processing or in statistical inference for fractional Gaussian processes.
 
The goal of this paper is twofold. Firstly, we investigate the tightness in the Breuer-Major theorem. Surprisingly, this problem did not receive a lot of attention until now, and the best available condition due to Ben Hariz \cite{BenHariz}
is neither arguably very natural, nor easy-to-check in practice.
In contrast, our condition very simple, as it only requires that $|f|^p$ must be integrable  with respect to the standard Gaussian measure for some $p$ \emph{strictly} bigger than 2. It is obtained by means of the Malliavin calculus, in particular Meyer inequalities.

Secondly, and motivated by a problem of geometrical nature, we extend the continuous Breuer-Major theorem to the notoriously difficult case of self-similar Gaussian processes which are \emph{not} necessarily stationary. An application to the fluctuations associated with the length process of a regularized version of the bifractional Browninan motion concludes the paper.
\normalsize
}

\section{Introduction and statement of the main results}
Let $Y=(Y(t))_{t\geq 0}$ be a zero-mean Gaussian stationary process, with covariance function 
$\E[Y(t) Y(s) ] =\rho(|t-s|)$ such that $\rho(0)=1$.  Let $\gamma =N(0,1)$ be the standard Gaussian measure on $\R$. Consider a function
$f \in L^2(\R,\gamma)$   of
Hermite rank $d \geq 1$, that is, $f$ has a
series expansion given by
\begin{equation}\label{hermite}
  f(x) = \sum_{q=d}^{\infty} c_q H_q(x), \,\,\, c_d\neq 0,
\end{equation}
where $H_q(x)$ is the $q$th Hermite polynomial.

It has become a central result in modern stochastic analysis that, under the condition $\int _\R  |\rho(s) |^d ds< \infty $, the \emph{finite-dimensional distributions} (f.d.d.) of the process
\begin{equation} \label{zeta}
Z_\e(t) := \sqrt{\e}  \int_0^{t/\e} f(Y(s))ds, \,\,  t\ge 0
\end{equation}
converge, as $\e$ tends to zero, to those of $\sigma W   $, where $W=(W(t))_{t\ge 0}$ is a standard Brownian motion
and  
\begin{equation} \label{sigma}
\sigma^2= \sum_{q=d} ^\infty c_q^2 q!  \int_\R \rho(s)^q  ds.
\end{equation}
(Observe that $|\rho(s)|=|\E[Y(s)Y(0)]|\leq \rho(0)=1$ by Cauchy-Schwarz, and thus $\sigma^2$ is
well defined under our integrabilility assumption on $\rho$ and the
square-integrability of $f$).
This is a continuous version of the celebrated Breuer-Major theorem proved in \cite {BM}, that can be found stated this way e.g. in the  paper  by Ben Hariz \cite{BenHariz}. 
We also refer the reader to \cite[Chapter 7]{IvanGioBook}, where a modern proof of the original discrete version\footnote{Note that the proof contained in \cite[Chapter 7]{IvanGioBook} can be easily  extended (mutatis mutandis) to cover the continuous framework as well.} of the Breuer-Major theorem is given, by means of the recent Malliavin-Stein approach.

The condition $\int _\R  |\rho(s) |^d ds< \infty $ turns out to be also necessary for the convergence of $Z_\e$ to $\sigma W$ in the sense of f.d.d., because $\sigma^2$ is not properly defined when  $\int _\R  |\rho(s) |^d ds= \infty $.
What about the \emph{functional convergence}, that is, convergence in law of $Z_\e$ to $\sigma W$ in  $C(\R_+)$ endowed with the uniform topology on compact sets?
First, let us note that Chambers and Slud \cite[page 328]{ChambersSlud} provide a counterexample of a zero-mean Gaussian stationary process $Y$ and a square-integrable function $f$ satisfying $Z_\e\Rightarrow \sigma W$ in the sense of f.d.d., but {\it not} in the functional sense; as a consequence, we see that the mere condition $\int _\R  |\rho(s) |^d ds< \infty $ 
does not imply tightness in general.

Before the present paper, the best sufficient condition ensuring tightness in the continous Breuer-Major theorem was due to Ben Hariz \cite{BenHariz}: more precisely, it is shown in \cite[Theorem 1]{BenHariz} that the functional convergence of $Z_\e$ to $\sigma W$ holds true  whenever either
\begin{equation}\label{bh1}
\mbox{there exists $R>1$
such that }\,
\sum_{q=d}^\infty \frac{|c_q|}{\sqrt{q!}}\left(\int_\R |\rho(s)|^q ds\right)^{1/2}\,R^q<\infty,
\end{equation}
or
\begin{equation}\label{bh2}
\mbox{the $c_q$ are all positive and $f\in L^4(\R,\gamma)$}.
\end{equation}
The two conditions (\ref{bh1})-(\ref{bh2}) proposed by Ben Hariz \cite{BenHariz} were obtained thanks to moment inequalities {\it \`a la} Rosenthal; they are neither very natural, nor easy-to-check. 

In the present paper, our first main objective is  to remedy the situation and provide a simple sufficient condition for the convergence $Z_\e\Rightarrow\sigma W$ to hold in  law in $C(\R_+)$ endowed with the uniform topology on compact sets. Surprisingly and compared to \cite{BenHariz}, our finding is that only a little more integrability of the function $f$ is enough.

\begin{theorem} \label{thm1}
Let $Y=(Y(t))_{t\geq 0}$ be a zero-mean Gaussian stationary  process with covariance function 
$\E[Y(t) Y(s) ] =\rho(|t-s|)$ such that $\rho(0)=1$. Consider a function  $f \in L^2(\R,\gamma)$ with expansion (\ref{hermite}) and Hermite rank $d\ge 1$. Suppose that  $\int _\R  |\rho(s) |^d ds< \infty $.  Then, if  $f \in L^p(\R,\gamma)$  for some $p>2$, 
the process $Z_\e$ defined in (\ref{zeta}) converges in   law in $C(\R_+)$ to $\sigma (W(t))_{t\ge 0}$, where $W$ is a Brownian motion and  $\sigma ^2$ is defined in (\ref{sigma}).
\end{theorem}

The proof of Theorem \ref{thm1}
is based on the  application of the techniques of Malliavin calculus and it has been inspired by  the recent work
of  Jaramillo and Nualart   \cite{JN} on the asymptotic behavior of the renormalized self-intersection local time of the fractional Brownian motion. The main idea  to prove tightness  is to  use the representation of the random variable $Z_\e(t)$ as 
\[
Z_\e(t)= \delta^d (-D L^{-1})^d Z_\e(t),
\]
where $\delta$, $D$ and $L$ are the basic operators in Malliavin calculus and then apply Meyer inequalities to upper bound
   $\E[ |Z_\e (t)- Z_\e(s) |^p]$ by $C |t-s| ^{p/2}$, where $p$ is the exponent appearing in  Theorem \ref{thm1}.

\medskip
Then, as an application of the previous result we aim to solve the following problem.
Let  $X= (X(t))_{t\ge 0}$ be a
   self-similar continuous Gaussian centered  process, and assume moreover that almost no path of $X$ is rectifiable, that is, the length of $X$ on any compact interval is \emph{infinite}: in symbols, $\mathcal{L}(X;[0,t])=+\infty$ for all $t>0$. Examples of such processes include
the fractional Brownian motion and relatives, such as the bifractional Brownian motion and  the subfractional Brownian motion.
Consider the $C^1$-regularization $X^\e$ of $X$  given by
\begin{equation}
X^\e(t)=\frac{1}{\e}\int_t^{t+\e}X(u)du.
\label{beps}
\end{equation}
Can we compute at which speed the length of $X^\e$ on $[0,t]$ explodes?
Stated in a different way, 
what is the asymptotic behaviour of the family of processes indexed by $\e$:
\begin{equation}\label{quest}
\mathcal{L}(X^\e;[0,t])=
\int_0^t \left|\dot{X}^\e(u)\right|du, \quad t\geq 0,
\end{equation}
when $\e\to 0$?

Let us first take a look at the simplest case, that is, where $X=B$ is a fractional Brownian motion (fBm)
of index $H\in(0,1)$. We recall that the fBm $B=(B(t))_{t\ge 0}$ is a centered Gaussian process with covariance
\begin{equation}
\label{fbm}
\E[B(t) B(s)] =  \frac 12 \big( t^{2H} + s^{2H} -|t-s|^{2H}\big).
\end{equation}
Making a change of variable and using  the self-similarity  of $B$ we observe that
\begin{eqnarray*}
\mathcal{L}(B^\e;[0,t]) &=&\e^{-1}\int_0^t |B(u+\e)-B(u)|du
=\int_0^{ t/\e} |B(\e(v+1))-B(\e v)|dv\\
&\overset{\rm law}{=}&\e^{H}\int_0^{ t/\e}|B(v+1)-B(v)|dv=:  \e ^{H-\frac 12} Z_\e(t)\quad\mbox{(as a process in $t$)},
\end{eqnarray*}
so that we are left to study the asymptotic behavior of $Z_\e$ as $\e\to 0$.
Since the fractional Gaussian noise $(B(t+1)-B(t))_{t \geq 0}$ is \emph{stationary}, 
to conclude it actually suffices to  apply Theorem \ref{thm1} to the process $Y(t)=B(t+1)-B(t)$.
 Indeed,  if we choose for $f$  the function $f(x)=\abs{x} - \sqrt{\frac{2}{\pi}}$ of Hermite rank 2 (indeed,
$f
= \sum_{q=2}^{\infty} \frac{1}{q! (2q-1)} H_{2q},
$
see Section \ref{s-12}), we have $f\in L^p(\R,\gamma)$ for any $p\ge 2$. 
 In this way we obtain   the following result:
\begin{itemize}
\item[(i)]  If $H< \frac34$, then
\begin{equation}\label{conv1}
\e^{\frac12-H}\left(\mathcal{L}(B^\e;[0,t]) -t\e^{H-1}\sqrt{\frac{2}{\pi}}\right)_{t\geq 0}\Rightarrow \sigma_H\,(W(t))_{t\geq 0}\quad\mbox{in $C(\R_+)$ as $\e\to 0$},
\end{equation}
\end{itemize}
 with $W$ a standard Brownian motion and
$
\sigma^2_H= \sum_{q=2}^{\infty} \frac{1}{q!  (2q-1)^2}  \int_{-\infty}^{\infty}  a_{2H}(h)^q dh,
$
  where, for any $\alpha >0$, 
  \begin{equation}
     \label{eq:6}
     a_{\alpha}(h)  =\frac{1}{2}
     \left(
       \abs{h-1}^{\alpha} + \abs{h+1}^{\alpha} - 2 \abs{h}^{\alpha}
     \right),  \,\,\, h\in \R.
   \end{equation}
   
   Furthermore,  in the case $H\ge \frac 34$, it is known that (tightness in the case $H=\frac 34$ can be proved by the same techniques as in Theorem \ref{thm1} and follows from Theorem \ref{thm:5} below):
\begin{itemize}
\item[(ii)] If $H=\frac34$, then 
\begin{equation}\label{conv2}
\frac{\e^{-\frac14}}{\sqrt{|\log \e|}}\left(\mathcal{L}(B^\e;[0,t]) -t\e^{-\frac14}\sqrt{\frac{2}{\pi}}\right)_{t\geq 0}\Rightarrow  \frac 18\,(W(t))_{t\geq 0}\quad\mbox{in $C(\R_+)$ as $\e\to 0$};
\end{equation}
\item[(iii)] If $H>\frac34$, then 
\begin{equation}\label{conv3}
\e^{H-1}\left(\mathcal{L}(B^\e;[0,t]) -t\e^{H-1}\sqrt{\frac{2}{\pi}}\right)\Rightarrow \mbox{`Rosenblatt process'}
\quad\mbox{in $C(\R_+)$ as $\e\to 0$}.
\end{equation}
\end{itemize}

The asymptotic behavior of (\ref{quest}) is therefore completely understood
when $X=B$ is a fBm.
But are the previous convergences (\ref{conv1}), (\ref{conv2}) and (\ref{conv3}) 
still true for any self-similar continuous  Gaussian centered process?  In this paper our second main objective is to answer this question, which is particularly difficult because of the lack of stationarity of the increments of $X$ in such a generality.

To have a better idea of what may happen, let us now consider the case where
$X=\widetilde{B}$ is the bifractional Brownian motion with indices $H\in(0,1)$ and $K\in (0,1]$, meaning that the covariance of $\widetilde{B}$ is given by
\begin{equation}
\label{bifbm}
\E[\widetilde{B}(t)\widetilde{B}(s)] = 2^{-K}\big((t^{2H}+s^{2H})^K-|t-s|^{2HK}\big).
\end{equation}
When $K=1$, $\widetilde{B}$ is nothing but a fBm with index $H$. In general, 
we can think of $\widetilde{B}$ as a perturbation of a fBm
$B$
with index $HK$. Indeed,
set $Z(t) =\int_0^\infty (1-e^{-\theta t})\theta^{-\frac{1+K}{2}}dW(\theta)$, $t\geq 0$, where $W$ stands for a standard Brownian motion independent of $\widetilde{B}$.
As shown by Lei and Nualart \cite{LN}, the process $Z$ has
absolutely continuous trajectories; moreover, with $Y(t)=Z(t^{2H})$,
\begin{equation}
\label{leinualart}
\left(\sqrt{\frac{2^{-K}K}{\Gamma(1-K)}}
\,Y(t)+ \widetilde{B}(t)\right)_{t\geq 0} \overset{\rm law}{=} \left(2^{\frac{1-K}{2}} B(t)\right)_{t\geq 0}.
\end{equation}
Recall definition (\ref{beps}). We immediately deduce from (\ref{leinualart}) that, for any $\e>0$,
\begin{equation}
\label{leinualart2}
\left(\sqrt{\frac{2^{-K}K}{\Gamma(1-K)}}
\,Y^\e(t)+ \widetilde{B}^\e(t)\right)_{t\geq 0} \overset{\rm law}{=} \left(2^{\frac{1-K}{2}} B^\e(t)\right)_{t\geq 0}.
\end{equation}
We can thus write, assuming that $Y$ and $\widetilde{B}$ are independent and defined on the same probability space, and with $B:=2^{\frac{K-1}{2}}(\sqrt{\frac{2^{-K}K}{\Gamma(1-K)}}
Y+\widetilde{B})$:
\begin{eqnarray*}
&&\mathcal{L}(\widetilde{B}^\e;[0,t]) - 2^{\frac{1-K}{2}} t \e^{HK-1} \sqrt{\frac{2}{\pi}}
=\e^{-1}\int_0^t |\widetilde{B}(u+\e)-\widetilde{B}(u)|du - 2^{\frac{1-K}{2}} t \e^{HK-1} \sqrt{\frac{2}{\pi}}\\
&=&2^{\frac{1-K}{2}}\int_0^t \left\{ 
\left|
\frac{B(u+\e)-B(u)}{\e}
\right|-\e^{HK-1}\sqrt{\frac{2}{\pi}}
\right\} du\\
&&  + \int_0^t \Bigg\{
\left|2^{\frac{1-K}{2}}\frac{B(u+\e)-B(u)}{\e}-\sqrt{\frac{2^{-K}K}{\Gamma(1-K)}}
 \frac{Y(u+\e)-Y(u)}{\e}\right|  \\
 && \qquad  - 2^{\frac{1-K}{2}}\left|\frac{B(u+\e)-B(u)}{\e}\right|
\Bigg\}du\\
&=:&a_\e(t)+b_\e(t).
\end{eqnarray*}
When $HK<\frac12$,  we deduce from (\ref{conv1}) that 
\begin{equation}\label{ae}
\e^{\frac12-HK}a_\e \Rightarrow 2^{\frac{1-K}{2}}\sigma_{HK}W,
\end{equation} 
whereas 
\begin{equation}\label{contrast}
|b_\e(t)|\leq \sqrt{\frac{2^{-K}K}{\Gamma(1-K)}}
\int_0^t \left|\frac{Y(u+\e)-Y(u)}{\e}\right| du\to \sqrt{\frac{2^{-K}K}{\Gamma(1-K)}}
\int_0^t
|\dot{Y}(u)|du.
\end{equation}
By combining (\ref{ae}) and (\ref{contrast}) together  we eventually obtain that 
\begin{equation}\label{conv1bis}
\e^{\frac12-HK}\left(\mathcal{L}(\widetilde{B}^\e;[0,t]) - 2^{\frac{1-K}{2}} t \e^{HK-1} \sqrt{\frac2\pi}\right)_{t\geq 0}  \Rightarrow \left(2^{\frac{1-K}{2}}\sigma_{HK}W(t)\right)_{t\geq 0},
\end{equation}
which is analogous to (\ref{conv1}).
The situation where $HK\geq \frac12$ looks more complicated at first glance because to conclude we not only need
an upper bound as the one given by (\ref{contrast}), but we  have
to understand the \emph{exact} behavior of $b_\e$ when $\e\to 0$.
For all $t>0$, one has almost surely that $\e^{-1}(Y(t+\e)-Y(t))\to \dot{Y}(t)= 2Ht^{2H-1}\dot{X}(t^{2H})$,
whereas $\e^{-1}|B(t+\e)-B(t)|$ diverges to $+\infty$. Hence, at a \emph{heuristic} level, one has that
\begin{multline*}
\left|2^{\frac{1-K}{2}}\frac{B(t+\e)-B(t)}{\e}-\sqrt{\frac{2^{-K}K}{\Gamma(1-K)}}
 \frac{Y(t+\e)-Y(t)}{\e}\right| - 2^{\frac{1-K}{2}}\left|\frac{B(t+\e)-B(t)}{\e}\right|
\\
\overset{\rm a.s.}{\to}
-\sqrt{\frac{2^{-K}K}{\Gamma(1-K)}}\,
\dot{Y}(t)\times \lim_{\e\to 0} {\rm sign}(B(t+\e)-B(t))=:A(t).
\end{multline*}
Although the previous reasoning is only heuristic (because  $\lim_{\e\to 0} {\rm sign}(B(t+\e)-B(t))$ does not exist), it seems to indicate that
$b_\e$ may converge almost surely as $\e\to 0$ \emph{without} further renormalization, to a random variable of the form
$\int_0^t A(u)du$.
If such a claim were  true, we would deduce from it that
\begin{multline}\label{claim}
\left(\mathcal{L}(\widetilde{B}^\e;[0,t])  - 2^{\frac{1-K}{2}} t \e^{HK-1}
  \sqrt{\frac2\pi}\right)_{t\geq0}
\\ \to
\begin{cases}
  2^{\frac{1-K}{2}}\sigma_{1/2} W + \int_0^\cdot A(u)du
  & \text{in law if $HK=\frac12$,}
\\
\int_0^\cdot A(u)du
& \text{a.s. if $HK > \frac{1}{2}$},
\end{cases}
\end{multline}
with $W$ a Brownian motion independent of
  $\widetilde{B}$,
a statement which would be very different compared to  (\ref{conv1}), (\ref{conv2}) and (\ref{conv3}).

Our first attempt to study the asymptotic behavior of (\ref{beps}) in the case where $X=\widetilde{B}$ is a bifractional Brownian motion  was 
to 
check whether the reasoning leading to (\ref{claim}) can be made rigorous. 
But we failed, to then realize that the claim (\ref{claim}) is actually wrong. What is correct is that convergences (\ref{conv1}), (\ref{conv2}) and (\ref{conv3}) 
continue to be valid for a wide class of self-similar centered Gaussian processes, containing not only the bifractional Brownian motion, but other perturbations of the fractional Brownian motion. 

With this application in mind,  the second goal of our paper is to generalize Theorem \ref{thm1}  to    self-similar Gaussian processes which are \emph{not} necessarily stationary. We will also  consider    the case  where the integral  $\int _\R  |\rho(s) |^d ds $ is infinite but the limit is still Gaussian (in such a critical case, a logarithmic factor is required), or when a non-Gaussian limit appears.

 Let us first present the class of processes under consideration.
Assume that 
 $X=(X(t))_{t \geq 0}$ is a centered Gaussian process that is self-similar of order 
$\beta \in (0,1)$. We define 
$\phi : [1,\infty) \to\mathbb{R}$ by
$
  \phi(x) = \Ex[X(1) X(x)]
$,
so that, for $0 < s \leq t$, we have
\begin{equation}
  \label{eq:27}
  \Ex[X(s) X(t)] = s^{2\beta} \Ex\left[X(1) X\left(\frac{t}{s}\right)\right]= s^{2\beta} \phi \left( \frac{t}{s} \right).
\end{equation}
Therefore $\phi$ characterizes the covariance function of $X$. 
Moreover, 
 let us also assume the following two hypotheses on $\phi$,
which were first introduced and considered in \cite{HN}:
\begin{enumerate}
\item[(H.1)] There exists $\alpha \in (0,2\beta]$ such that $\phi$ has the form
  \begin{equation*}
    \phi(x) = -\lambda (x-1)^{\alpha} + \psi(x),
  \end{equation*}
  where $\lambda>0$ and $\psi(x)$ is twice-differentiable on an open set containing
  $[1,\infty)$ and there exists a constant $C \geq 0 $ such that, for any $x \in (1,\infty)$,
  \begin{enumerate}
  \item $\abs{\psi'(x)}  \leq C x^{\alpha-1}$
  \item $\abs{\psi''(x)} \leq C x^{-1} (x-1)^{\alpha-1}$
  \item $\psi'(1) = \beta \psi(1)$ when $\alpha \geq 1$.
  \end{enumerate}
\item[(H.2)] There are constants $C > 0$, $c>1$ and $1 < \nu \leq 2$ such that, for all $x \geq
  c$,  
\begin{enumerate}
  \item[(d)] $\abs{\phi'(x)} \leq
    \begin{cases}
      C x^{-\nu} &\qquad \text{if $\alpha<1$,} \\
      C x^{\alpha-2} &\qquad \text{if $\alpha \geq 1$}.
    \end{cases}
$
\item[(e)] $\abs{\phi''(x)} \leq
      \begin{cases}
      C x^{-\nu-1} &\qquad \text{if $\alpha<1$,} \\
      C x^{\alpha-3} &\qquad \text{if $\alpha \geq 1$}.
    \end{cases}$
\end{enumerate}
\end{enumerate}

We refer to \cite[Section 4]{HN}
for explicit examples of processes $X$ satisfying (H.1) and (H.2), among them the bifractional Brownian motion (\cite[Section 4.1]{HN}) and the subfractional Brownian motion (\cite[Section 4.2]{HN}).

Now, for $\e >0$ and $t \geq 0$, let us define
\begin{equation}
  \label{eq:1}
  \Delta_\e X(t) = X(t+\e) - X(t) \qquad \text{and} \qquad
  Y_{\e}(t) = \frac{\Delta_\e X(t)}{\norm{\Delta_\e X(t)}_{L^2(\Omega)}}.
\end{equation}
Finally, define the family of stochastic processes $\widetilde{F}_\e= (\widetilde{F}_\e(t))_{t \geq 0}$ by
  \begin{equation}
    \label{eq:2}
    \widetilde{F}_\e (t)= \frac1{\sqrt{\e}} \int_0^{t} f(Y_\e(u))du.
  \end{equation}
  By the  self-similarity property of $X$,  the process $\widetilde{F}_\e$ has the same law as $F_\e$, where
  \begin{equation}
    \label{eq:2a}
    F_\e (t)= {\sqrt{\e}}   \int_0^{ t/\e} f(Y_1(u))du.
  \end{equation}

The  second  contribution of this paper is the following theorem, which is an
extension of Theorem \ref{thm1} (central case) and the main results of Taqqu's
seminal paper \cite{T} (non-central case) to a situation where the underlying Gaussian process $X$  does  \emph{not}  need to have stationary increments. 
This lack of stationarity is actually the main difficulty we will have to cope with.   

\begin{theorem}
  \label{thm:5}
  In the above setting, assume that (H.1) and (H.2) are in order for 
a centered Gaussian process $X=(X(t))_{t \ge 0}$, self-similar of order 
$\beta \in (0,1)$, and whose covariance function is given by (\ref{eq:27}).
Let $f\in L^2(\R,\gamma)$  a function with Hermite rank $d\ge 1$ and expansion (\ref{hermite}) and  let  $F_\e$ be defined in (\ref{eq:2a}).
 Then the following is true as $\e\to 0$.
  \begin{enumerate}
  \item If $\alpha < 2-\frac{1}{d}$, then the finite-dimensional  distributions of the family $\left\{ F_\e :\, \e>0 
    \right\}$ converges in law  to those of a
    Brownian motion with variance given by (\ref{sigma}) with  $\rho(h) = a_{\alpha}(h)$ defined in (\ref{eq:6}).
  \item If $\alpha = 2-\frac{1}{d}$, then the finite-dimensional  distributions of the family
    $\left\{ F_{\e}/\sqrt{|\log \e|} :\, \e>0 \right\}$ converges in law  to those of a
    Brownian motion  with variance 
    \begin{equation}
      \label{eq:69}
     \sigma^2_{1- \frac 2d}=  c_d^2\, d! 
      \left(1 + (\beta - \frac{\alpha}2)d\right)
       \left( 1 - \frac{1}{2d} \right)^d \left( 1 - \frac{1}{d} \right)^d.
    \end{equation}
    \end{enumerate}
    Moreover, if $f\in L^p(\R ,\gamma)$ for some $p>2$, then the convergences in  1. and 2. hold in law in $C(\R_+)$.
\end{theorem}
    
    Let us finally consider the case
   $\alpha > 2 - \frac{1}{d}$ and $d\ge 2$.
   We will show that, for all $t\ge 0$,  the random variable  $\e^{\frac 12 -d( 1-\frac \alpha 2)} \widetilde{F}_{\varepsilon}(t)$
converges in $L^2(\Omega)$ to a random variable   $c_d H_\infty(t)$ belonging to the $d$th Wiener chaos.
    The process $H_\infty = (H_\infty(t))_{t\ge 0}$ is a generalization of the Hermite process (see \cite{DM,NNT,T}) and it has a covariance given by
     \begin{equation}
    \label{eq:70a}
   K_d(s,t)=   \Ex[H_{\infty}(s) H_{\infty}(t)] = 
    \frac{ d!}{\left( 2\lambda \right)^d}
    \int_0^{s} \int_0^{t} \left(
      \frac{
        \partial_{u}\partial_v \Ex[X(u) X(v)]
      }{
        \left( uv \right)^{\beta-\alpha/2}
      }
    \right)^d  du dv.
  \end{equation}
  This then leads to the following  non-central limit theorem
    in the case   $\alpha > 2 - \frac{1}{d}$. 
    
    \begin{theorem} \label{thm1.3}
 Under the assumptions of  Theorem \ref{thm:5},  if $\alpha > 2 - \frac{1}{d}$, then the process
 $\left\{ \e^{\frac12-d(1-\frac{\alpha}2) }
    F_\e :\, \e>0 \right\}$ converges in law   in $C(\R_+)$ to $F_{\infty}= c_d H_\infty$.
\end{theorem}

We note that a \emph{discrete} counterpart of point 1 in Theorem 1.1 was already obtained by Harnett and Nualart in \cite{HN}, in exactly the same setting.  However, we would like to offer the following comments, to help the reader comparing our results with those contained in \cite{HN}.  Firstly,   neither point 2  of Theorem  \ref{thm:5} nor   the tightness property  and  Theorem \ref{thm1.3} have been considered in \cite{HN}.   Secondly, and a little bit against common intuition, it turns out that it was more
difficult to deal with the continuous setting; indeed, in the continuous case we have to handle the situation  where $|t-s| <1$, which does not appear in the discrete setting.
Thirdly, our original motivation of proving Theorems \ref{thm1}, \ref{thm:5} and  \ref{thm1.3} is of geometrical nature; in our mind, this work
actually represents a first step towards a better understanding of
the asymptotic behavior of functionals
of the kind (\ref{quest}) (or more complicated ones) that arise very often in differential geometry.

To conclude this introduction, let us go back to the case of the bifractional
Brownian motion $X=\widetilde{B}$, and let us see what the conclusions of
Theorems \ref{thm:5} and \ref{thm1.3} become in this case, when for $f$ we choose the function $f(x)=\abs{x} - \sqrt{\frac{2}{\pi}}$. Since, on one hand, the
bifractional Brownian motion defined by (\ref{bifbm}) satisfies (H.1) and (H.2) with $\alpha=2\beta=2HK$ and, on the other hand, 
one has $\|\Delta_\e \widetilde{B}(t)\|_{L^2(\Omega)}\sim 2^{\frac{1-K}{2}}\e^{HK}$
as $\e\to 0$ for any $t>0$,   we deduce from our Theorems~\ref{thm:5} and \ref{thm1.3}  that:
  \begin{itemize}
  \item if $HK< 3/4$, then the family $\left\{ \e^{\frac12-HK}\left( \mathcal{L}(\widetilde{B}^\e;[0,t])
  -\E\left[ \mathcal{L}(\widetilde{B}^\e;[0,t])\right]\right):\, \e>0 \right\}$ converges in law in $C(\R_+)$
    to a Brownian motion with variance
  \begin{equation*}
2^{1-K} \sum_{q=2}^{\infty} \frac{1}{q!  (2q-1)^2}  \int_{-\infty}^{\infty}  a_{2HK}(h)^q dh,
  \end{equation*}
  see also (\ref{conv1bis}) and compare with claim (\ref{claim});
\item if $HK=3/4$, then the family $\left\{ \frac{\e^{-\frac14}}{\sqrt{\log \e}}\left( \mathcal{L}(\widetilde{B}^\e;[0,t])
  -\E\left[ \mathcal{L}(\widetilde{B}^\e;[0,t])\right]\right):\, \e>0 \right\}$ converges in law in $C(\R_+)$ to a Brownian motion with variance $2^{-K}/64$;
\item if $HK > 3/4$, then the family $\left\{ \e^{HK-1}\left( \mathcal{L}(\widetilde{B}^\e;[0,t])
  -\E\left[ \mathcal{L}(\widetilde{B}^\e;[0,t])\right]\right):\, \e>0 \right\}$  converges in law in $C(\R_+)$ towards a stochastic process $F_{\infty}$ which lies in
the second Wiener chaos.
\end{itemize}

The rest of the paper is organized as follows.   Section \ref{sec2} contains some preliminaries on Maliavin calculus and a basic multivariate chaotic central limit theorem.   The proof of Theorem \ref{thm1} is  given in Section \ref{sec3}. Section \ref{sec4}  provides some useful properties satisfied by  self-similar processes $X$ under assumptions (H.1) and (H.2) and contains the proof of Theorem
\ref{thm:5}. The proof of Theorem
 \ref{thm1.3} is then given in Section  \ref{sec5}. Finally, Section \ref{sec6} contains some technical lemmas that are used along the paper. 

   Throughout the paper, $C$ denotes a  generic positive constant  whose value
   may change from line to line.

\section{Preliminaries}  \label{sec2}
In this section, we gather several preliminary results that will be used for
proving the main results of this paper. 

\subsection{Elements of Malliavin calculus}

We assume that the reader is  already familiar with the classical concepts of
Malliavin calculus, as outlined for example in the three books \cite{IvanGioBook,DavidBook,DavidEulaliaBook}.

To be in a position to use Malliavin calculus to prove the results of our paper,
we shall adopt the following classical Hilbert space notation. 
Let   $\HH$ be a real and separable Hilbert space.  Let  $X$ be an {\it  isonormal} Gaussian process indexed by $\HH$ and defined
on a probability space $(\Omega,\mathcal{F}, \mathbb{P})$, that is, $X=\{X(h), h\in\mathfrak{H}\}$ is a family of jointly centered Gaussian random variables satisfying $\E[X(h)X(g)]=\langle h,g\rangle_\HH$ for all
$h,g\in \HH$. We will also assume that $\mathcal{F}$ is the $\sigma$-field  generated by $X$.

For integers $q\geq 1$, let $\HH^{\otimes q}$ denote the $q$th tensor product of
$\HH$, and let $\HH^{\odot q}$ denote the subspace of symmetric tensors of
$\HH^{\otimes q}$. 
Let $\{e_n\}_{n\geq 1}$ be a complete orthonormal system in $\HH$.
For functions $f,g\in\HH^{\odot q}$ and $r\in\{1,\ldots,q\}$ we define
the $r$th-order contraction of $f$ and $g$ as the element
of $\HH^{\otimes(2q-2r)}$ given by
$$
f\otimes_r g = \sum_{i_1,\ldots,i_r=1}^\infty 
\langle f,e_{i_1}\otimes \ldots\otimes e_{i_r}\rangle_{\HH^{\otimes r}}
\langle g,e_{i_1}\otimes \ldots\otimes e_{i_r}\rangle_{\HH^{\otimes r}},
$$
where $f\otimes_0 g=f\otimes g$ by definition and, if $f,g\in\HH^{\odot q}$, $f\otimes_q g=\langle f,g\rangle_{\HH^{\otimes q}}$.

The $q$th Wiener chaos is the closed linear subspace
of $L^2(\Omega)$ that is generated by the random variables $\{H_q(X(h)),\,h\in \HH,\,\|h\|_\HH=1\}$, where $H_q$ stands for the $q$th Hermite polynomial.
For $q\geq 1$, it is known that the map $I_q(h^{\otimes q})=H_q(X(h))$ ($h\in\HH$, $\|h\|_\HH=1$) provides a linear isometry between  $\HH^{\odot q}$ (equipped with the modified 
norm $\sqrt{q!}\|\cdot\|_{\HH^{\otimes q}}$) and the $q$th Wiener chaos. By convention, $I_0(x)=x$ for all $x\in\R$.

It is well-known that any $F\in L^2(\Omega)$ can be decomposed into Wiener chaos as follows:
\begin{equation}\label{decompo}
F=\E[F]+\sum_{q=1}^\infty I_q(f_q),
\end{equation}
where the kernels $f_q\in \HH^{\odot q}$ are uniquely determined by $F$.

For a smooth and cylindrical random variable $F= f(X(h_1), \dots , X(h_n))$,
with $h_i \in \mathfrak{H}$ and $f \in C_b^{\infty}(\mathbb{R}^n)$ ($f$ and all of its partial derivatives are bounded), we define its Malliavin derivative as the $\mathfrak{H}$-valued random variable given by
\[
 DF = \sum_{i=1}^n \frac{\partial f}{\partial x_i} (X(h_1), \dots, X(h_n))h_i.
\]
By iteration, one can define the $k$-th derivative $D^k F$  as an element of $L^2(\Omega; \mathfrak{H}^{\otimes k})$. For any natural number $k$ and any real number $ p \geq 1$, we define  the Sobolev space $\mathbb{D}^{k,p}$  as the closure of the space of smooth and cylindrical random variables with respect to the norm $\|\cdot\|_{k,p}$ defined by 
\[
 \|F\|^p_{k,p} = \mathbb{E}(|F|^p) + \sum_{i=1}^k \mathbb{E}(\|D^i F\|^p_{\mathfrak{H}^{\otimes i}}).
\]
The divergence operator $\delta$ is defined as the adjoint of the derivative operator $D$.  An element $u \in L^2(\Omega; \mathfrak{H})$ belongs to the domain of $\delta$, denoted by ${\rm Dom}\, \delta$, if there is a constant $c_u$ depending on $u$ such that 
\[
|\mathbb{E} (\langle DF, u \rangle_{\mathfrak{H}})| \leq c_u \|F\|_{L^2(\Omega)}
\] for any $F \in \mathbb{D}^{1,2}$.  If $u \in {\rm Dom} \,\delta$, then the random variable $\delta(u)$ is defined by the duality relationship 
\begin{equation} \label{dua}
\mathbb{E}(F\delta(u)) = \mathbb{E} (\langle DF, u \rangle_{\mathfrak{H}}) \, ,
\end{equation}
which holds for any $F \in \mathbb{D}^{1,2}$.  
In a similar way we can introduce the iterated divergence operator $\delta^k$ for each integer $k\ge 2$, defined by the duality relationship 
\begin{equation} \label{dua2}
\mathbb{E}(F\delta^k(u)) = \mathbb{E}  \left(\langle D^kF, u \rangle_{\mathfrak{H}^{\otimes k}} \right) \, ,
\end{equation}
for any $F \in \mathbb{D}^{k,2}$, where $u\in  {\rm Dom}\, \delta^k \subset L^2(\Omega; \mathfrak{H}^{\otimes k})$.

The Ornstein-Uhlenbeck semigroup $(T_t)_{t \geq 0}$ is the  semigroup of operators on $L^2(\Omega)$ defined by
\[
T_t F = \sum_{q=0}^\infty e^{-qt} I_q(f_q),
\]
if $F$ admits the Wiener chaos expansion (\ref{decompo}). Denote by $L $ the infinitesimal generator of $(T_t)_{t \geq 0}$ in $L^2(\Omega)$. Let   $L^{-1} F = -\sum_{q=1}^\infty \frac{1}{q} I_q(f_q) $ if $F$ is given by  (\ref{decompo}).   

The operators $D$, $\delta $ and $L$ satisfy the relationship $L=-\delta D$, which leads to the representation
\begin{equation}  \label{deltad}
F= -\delta DL^{-1}F,
\end{equation}
for any centered random variable $F\in L^2(\Omega)$.

Consider the isonormal Gaussian process $X(h)=h$  indexed by $\HH=\R$,  defined in the probability space $(\R, \mathcal{B}(\R), \gamma)$.  We denote the corresponding Sobolev spaces of functions by $\mathbb{D} ^{k,p} (\R,\gamma)$. In this context,  for any function
$g$, we have $Dg =g'$,  $\delta g= xg-g'$ and  $Lg=g''-xg' $ (see \cite{IvanGioBook}).
Let  $f\in L^2( \R, \gamma)$  be a function of Hermite rank $d$, with  expansion  (\ref{hermite}).   Let us introduce  the function $f_d$ defined by a shift of $d$ units in the coefficients, that is,
\begin{equation} \label{fd3}
f_d(x)= \sum_{q=d} ^\infty c_q H_{q-d}(x).
\end{equation}
We claim that $f_d$ belongs to the Sobolev space $\mathbb{D}^{2,d} (\R,\gamma)$.  In fact,  using that $H_q'(x)= qH_{q-1}(x)$, we can write
\[
f^{(d)}_d(x) = \sum_{q=2d} ^\infty c_q (q-d) (q-d-1) \cdots (q-2d+1) H_{q-2d}(x)
\]
and
\begin{align*}
\| f^{(d)} _d \|^2_{L^2(\R,\gamma)}
 = \sum_{q=2d}^ \infty c_q^2  (q-d) ^2(q-d-1)^2\cdots (q-2d+1)^2 (q-2d)! 
\le  \sum_{q=2d} ^\infty c_q^2  q!< \infty.
\end{align*}
 The function $f_d$ has the following representation in terms of the Malliavin operators:
\begin{equation} \label{fd}
f_d= (-DL^{-1})^df.
\end{equation}
Indeed, using that $H_q'(x)= qH_{q-1}(x)$, we have
\[
-DL^{-1} f=\sum_{q=d} ^\infty  \frac {c_q} q H'_{q}(x)= \sum_{q=d} ^\infty  c_q  H_{q-1}(x),
 \]
 and iterating $d$ times this formula, we get (\ref{fd}). 
  Formula (\ref{fd}) implies that if $f\in L^p (\R, \gamma)$ for some $p>1$, then  $f_d \in  \mathbb{D}^{d,p} (\R, \gamma)$, that is, $f_d$ is
 $d$-times weakly differentiable with  derivatives in $L^p(\R,\gamma)$. 
  In fact, by Meyer inequalities (see \cite{DavidBook}), the operators $D$ and   $(-L)^{1/2}$ are equivalent in  $L^p(\R, \gamma)$  and we obtain, for any $k=1,\dots, d$
\begin{equation}  \label{equ1}
\|f_d^{(k)}\|_{L^p(\R, \gamma)}  = \|  D^k[(-DL^{-1})^d f] \|_{L^p(\R, \gamma)} \le c_p  
\|   (-L)^{k/2} (DL^{-1})^d f \|_{L^p(\R, \gamma)} \le c_p   \|f\| _{L^p(\R, \gamma)}.
\end{equation}

\subsection{Multivariate chaotic central limit theorem}
Points 1 and 2 in Theorem \ref{thm:5} will be obtained by checking that the assumptions of the following theorem are satisfied.
We assume that  $X$ is an  isonormal  Gaussian process indexed by $\HH$ defined
on a probability space $(\Omega,\mathcal{F}, \mathbb{P})$, with $\mathcal{F}$ is the  $\sigma$-field generated by $X$.

\begin{theorem}
  \label{thm:4}
Fix  an integer $p\geq 1$, and  let $\left\{ G^\e:\,\e>0 \right\}$ be a family of $p$-dimensional vectors with components in $L^2(\Omega)$ and centered. According to (\ref{decompo}), we can write each component $G_i^\e$ of $G^\e$ in the form
  \begin{equation*}
    G_i^\e = \sum_{q=1}^{\infty} I_q(g^{\e}_{i,q}).
  \end{equation*}
  Let us suppose that the following conditions hold.
  \begin{enumerate}
  \item[(a)] For each $i,j\in\{1,\ldots,p\}$ and each $q \geq 1$, $ \sigma_{i,j,q}=\lim_{\e \to 0} q! \langle g^\e_{i,q},g^\e_{j,q}\rangle_{\HH^{\otimes q}}$ exists.
  \item[(b)]  For each $i\in\{1,\ldots,p\}$, $\sum_{q=1}^{\infty} \sigma_{i,i,q} < \infty$.
  \item[(c)]  For each $i\in\{1,\ldots,p\}$, each $q \geq 2$ and each $r = 1,\dots,q-1$, \break  $\lim_{\e \to 0}  \|      g^\e_{i,q} \otimes_r
  g^\e_{i,q}  \| _{\HH^{\otimes {2q-2r}}}=0$.
  \item[(d)]  For each $i\in\{1,\ldots,p\}$, $\lim_{N \to \infty} \sup_{\e \in(0,1]} \, \sum_{q=N+1}^{\infty} q!
       \| g^\e_{i,q} \| _{\HH^{\otimes {2q}}}^2 = 0$.
  \end{enumerate}
Then $G^\e=(G_1^\e,\ldots,G_p^\e)$ converges in distribution to
$N_p(0,\Sigma)$ as $\e$ tends to zero, where $\Sigma=(\sigma_{i,j})_{1\leq i,j\leq p}$ is defined by
$\sigma_{i,j}=\sum_{q=1}^\infty \sigma_{i,j,q}$.
\end{theorem}
\noindent
{\it Proof}.  This theorem is a multivariate counterpart of the chaotic central limit theorem proved  by  Hu and Nualart in \cite{HuNu}.
First notice that, by the results of Nualart and Peccati \cite{NuPe} and Peccati and Tudor \cite{PeTu},  conditions (a) and (c) imply that,
for any $N\ge 1$,  the family of random vectors  $(I_q(g^\e_{i,q}))_{1\le q\le N, 1\le i\le p}$ converges in law to a  centered Gaussian vector
 $(Z_{i,q})_{1\le q\le N, 1\le i\le p}$ with covariance $\E[Z_{i,q}Z_{j,q'}]=\sigma_{i,j,q} \delta_{q,q'}$. This implies that, for each $N\ge 1$, 
 the   family  of $p$-dimensional random vectors $\left(\sum_{q=1} ^N  I_q(g^\e_{i,q} )\right)_{1\le i   \le p}$ converges in law to  the Gaussian distribution $N_p(0, \Sigma^N)$, where $\Sigma^N=(\sigma^N_{i,j})_{1\leq i,j\leq p}$ is defined by
$\sigma^N_{i,j}=\sum_{q=1}^N \sigma_{i,j,q}$. Finally, conditions (b) and (d) and a simple triangular inequality allows us to conclude the proof. 
\qed

\section{Proof of Theorem~\ref{thm1}}
\label{sec3}

Since the convergence in the sense of f.d.d. follows from the classical Breuer-Major theorem (see, e.g., \cite{BenHariz}),
it remains to show that the family $\{ Z_\e : \e >0\}$ is tight.
For this we need to show that for any $0\le s<t  $ and $\e >0$ and for some $p>2$, there exists a constant $C_{p}>0$ such that
\[
\| Z_\e(t) - Z_\e(s) \|_{L^p(\Omega)}  \le C_{p} |t-s|^{1/2}.
\]
To show this inequality we will use an approach based on stochastic integral representations and Meyer's inequalities.
   
 Let $\mathfrak{H}$ be the Hilbert space defined as the closure of the set of step functions with respect to the scalar product
$\langle {\bf 1}_{[0,t]},{\bf 1}_{[0,s]}\rangle_\mathfrak{H} = \E[Y(s)Y(t)]$, $s,t\geq 0$. 
By identifying $Y(t)$ with $Y({\bf 1}_{[0,t]})$,  we can thus suppose that $Y$ is an isonormal Gaussian process indexed by $\HH$  defined on a probability space $(\Omega, \mathcal{F} , \mathbb{P})$. We will assume that $\mathcal{F}$ is generated by $Y$.

The function $f_d$ introduced in (\ref{fd3})  leads to  the following representation of $f(Y(u))$ as an iterated divergence:
\[
f(Y(u))= \delta^d \left( f_d(Y(u))     \mathbf{1}_{[0,u]} ^{\otimes d} \right).
\]
Indeed,
\begin{align*}
f(Y(u))&=\sum_{q=d} ^\infty c_q H_{q}(Y(u)) =\sum_{q=d} ^\infty c_q I_q\left(
   \mathbf{1}_{[0,u]} ^{\otimes q} \right)\\
&
=
\sum_{q=d} ^\infty c_q  \delta^d \left( I_{q-d} \left(
    \mathbf{1}_{[0,u]} ^{\otimes q-d} \right)     \mathbf{1}_{[0,u]} ^{\otimes d}
\right) 
= \delta^d \left( \sum_{q=d} ^\infty c_q  H_{q-d} \left(
Y(u)  \right)     \mathbf{1}_{[0,u]}^{\otimes d}
\right).
\end{align*}
Then,   using Meyer's inequalities, we obtain
\begin{align*}
\| Z_\e(t)- Z_\e(s) \| _{L^p(\Omega)} &
=\sqrt{\e}  \left\| \int_{s/\e} ^{t/\e} f(Y(u)) du \right \| _{L^p(\Omega)} \\
&= \sqrt{\e}  \left\| \int_{s/\e} ^{t/\e}  \delta ^d \left( f_d(Y(u))       \mathbf{1}_{[0,u]} ^{\otimes d} \right)d u\right \| _{L^p(\Omega)} \\
&\le 
 c_p    \sum_{k=0}^d  \sqrt{\e} \left\| \int_{s/\e} ^{t/\e}  D^k \left( f_d(Y(u) )    \mathbf{1}_{[0,u]} ^{\otimes d} \right)du \right \| _{L^p(\Omega; \HH^{\otimes k})} 
\\ &=:  c_p   \sum_{k=0}^d  R_k.
\end{align*}
   Using Minkowski and H\"older inequalities, we can write, for any $k=0,1, \dots, d$,
\begin{align*}
R_k  &=
 \sqrt{\e}  \left\| \int_{ [s/\e,t/\e]^2} f^{(k)}_d(Y(u))    f^{(k)}_d(Y(v))    \langle      \mathbf{1}_{[0,u]},    \mathbf{1}_{[0,v]}  \rangle_{\HH}  ^{d+k}  du dv\right \| _{L^{p/2}(\Omega)} ^{1/2}\\
& \le 
    \| f^{(k)}_d \|_{L^p(\R, \gamma)}  \left( \e \int_{s/\e} ^{t/\e}  \int_{s/\e} ^{t/\e}     | \rho(u-v)|^{d+k} dudv   \right) ^{1/2}.
\end{align*}
Using the assumptions of Theorem \ref{thm1} as well as (\ref{equ1}), we deduce that  $ \| f^{(k)}_d \|_{L^p(\R, \gamma)}$  is finite.
Finally,  the change of variable $(u,v) \to (u,u+h)$ leads to
 \[
 \e \int_{s/\e} ^{t/\e}  \int_{s/\e} ^{t/\e}     |    \rho(u-v)|^{d+k} dudv
 \le   C (t-s) \int_{\R} | \rho(h)|^{d+k} dh \leq  C(t-s),
 \]
 which provides the desired estimate.

\section{Proof of Theorem~\ref{thm:5}}
\label{sec4}

In this section $X$ will be a self-similar Gaussian process with covariance   (\ref{eq:27}). 
Before we proceed to the proof of  Theorem~\ref{thm:5}, we will show three technical lemmas which  provide   information on  the variance and covariance
of  $X$  under  Hypotheses (H.1) and (H.2).

\subsection{A few useful properties satisfied by $X$}
The first lemma  lemma give the structure of the variance  of an increment of length one, assuming Hypothesis   (H.1).

\begin{lemma}
   \label{lem:2}
   Assuming (H.1), there exists a continuous function $u_1:[0,\infty)\to\R$ such that for $s
   \geq 0$ 
   \begin{equation*}
     \Ex\big[\left( X(s+1)-X(s) \right)^2\big] = 2\lambda s^{2\beta-\alpha}  \left( 1 + u_1(s) \right).
   \end{equation*}
   Furthermore, given $\eta  >0$, there exists a positive constant $C_{\eta}$ such
   that for all $s\geq \eta$ one has
   \begin{equation}
     \label{eq:14}
     \abs{u_1(s)} \leq C_{\eta} \, s^{-\delta_1},\quad
\mbox{where $\delta_1=
\left\{
\begin{array}{ll}   
   1-\alpha&\quad\mbox{if $\alpha<1$}\\
   2-\alpha&\quad\mbox{if $\alpha \ge 1$}
   \end{array}
   \right.
   $}.
      \end{equation}
\end{lemma}

\begin{proof}
  If $s \ge 1$, the assertion follows from \cite[Lemma 3.1]{HN}. Let us now assume that
  $0<s<1$. Proceeding as in the proof of \cite[Lemma 3.1]{HN},
  we get that 
  \begin{align*}
    \Ex\big[ \left( X(s+1) - X(s) \right)^2\big]
    &=
      2\lambda s^{2\beta-\alpha}
      + \psi(1) \left( \left( s+1 \right)^{2\beta} - s^{2\beta} \right)
      - 2s^{2\beta} \int_1^{1+\frac{1}{s}} \psi'(y) dy
    \\ &=
         2\lambda s^{2\beta-\alpha} \left( 1 + u_1(s) \right),
  \end{align*}
  where
  \begin{equation*}
    u_1(s)
    = { (2\lambda)^{-1}}
     \psi(1) s^{\alpha-2\beta} \left( \left( s + 1 \right)^{2\beta} - s^{2\beta} \right)
      -  { \lambda^{-1} } s^{\alpha} \int_1^{1+\frac{1}{s}} \psi'(y) dy.
    \end{equation*}
  Then the bound~\eqref{eq:14} for $\eta \le s<1$ follows  immediately from the fact that 
  $u_1(s)$ is bounded for $\eta \le s<1$. 
      \end{proof} 

In the  next two lemmas  we will show formulas and estimates for the covariance  $\Ex[( X(t+1) - X(t) ) ( X(s+1) - X(s) )]$ in two different situations. First, we will assume Hypothesis (H.1) and consider the case where  $|t-s| \le M_1 (s\wedge t) +M_2$ for some constants $M_1$ and $M_2$, and  the second lemma  will handle the case   $|t-s| \ge (c-1) (s\wedge t) +c$  under Hypotheses (H.1) and (H.2), where $c$ is the constant appearing in (H.2).

\begin{lemma}
   \label{lem:1}
   Assume (H.1) and let $\eta > 0$. Then, for all $s,t>\eta$ satisfying $\eta \leq \abs{s-t}\le M_1
   \left( s \land t \right) + M_2$ for some positive constants $M_1$, $M_2$, it holds that
   \begin{equation*}
     \Ex\big[
       \left( X(t+1) - X(t) \right)
       \left( X(s+1) - X(s) \right)
     \big]
     =
     \lambda \left( s \land t \right)^{2\beta-\alpha}
     \left(
      {2} a_{\alpha}(s-t) + u_2(s,t)
     \right),
   \end{equation*}
   where $a_{\alpha}(h)$ is the function defined in (\ref{eq:6})
   and $u_2 :[0,\infty) \rightarrow \R$ is a continuous function satisfying the bounds
   \begin{equation}
     \label{eq:19}
     \abs{u_2(s,t)} \leq C
     \left(
       \left( s \land t \right)^{-1}  \abs{s-t}^{\alpha-1}
       +
       \left( s \land t \right)^{\alpha-2}
     \right) \quad  {\rm if} \quad  |s-t| \ge 1
   \end{equation}
  and
   \begin{equation}
     \label{eq:19a}
     \abs{u_2(s,t)} \leq C
     \left(
       \left( s \land t \right)^{\alpha -1}  \mathbf{1}_{\{\alpha <1\}}
       +
       \left( s \land t \right)^{\alpha-2}  \mathbf{1}
       _{\{ \alpha \ge 1 \}}
     \right) \quad  {\rm if} \quad  |s-t| <1.
   \end{equation}
\end{lemma}

 \begin{proof} 
  Assume without loss of generality that $t > s$, so that $t=s+h$ for some
  $h>0$. Then the assertion becomes
   \begin{equation*}
       \Ex\big[\left( X(s+h+1) - X(s+h) \right) \left( X(s+1)-X(s) \right)\big]
       =  s^{2\beta-\alpha} \lambda \left( 2a_\alpha (h) + u_{2}(s,s+h)  \right),
     \end{equation*}
     where $u_2(s,s+h)$ satisfies the bounds
     \begin{equation}
       \label{eq:24}
       \abs{u_2(s,s+h)} \leq C
     \left(
       s^{-1}  h^{\alpha-1}
       +
       s^{\alpha-2}
     \right)
     \end{equation}
     for all $s,h$ such that $s \ge \eta$ and $1 \le h \leq M_1 \, s + M_2$ and
     \begin{equation}
       \label{eq:24a}
       \abs{u_2(s,s+h)} \leq C
     \left(
       s^{\alpha-1}   \mathbf{1}_{\{ \alpha <1\}}
       +
       s^{\alpha-2} \mathbf{1}_{\{\alpha \ge 1\}}
     \right)
     \end{equation}
     for all  $s,h$ such that $s \ge \eta$ and $\eta \le h <1$.
     
     Let us first show the claim  (\ref{eq:24a}).   In this case,
  \begin{align*}
&   \Ex\big[\left( X(s+h+1)-X(s+h) \right)\left( X(s+1) - X(s) \right)\big]  \\
      &\qquad =   \left( s+1 \right)^{2\beta} \phi \left( \frac{s+h+1}{s+1} \right) 
  -
      \left( s+h \right)^{2\beta} \phi \left( \frac{s+1}{s+h} \right) \\
    &  \qquad   \qquad    -
      s^{2\beta} \phi \left( \frac{s+h+1}{s} \right)
      +
      s^{2\beta} \phi \left( \frac{s+h}{s} \right) \\
     & \qquad =
         - \lambda
         \left(
         \left( s+1 \right)^{2\beta-\alpha} h^{\alpha}
         -
         \left( s+h \right)^{2\beta-\alpha} \left( 1-h \right)^{\alpha}
         -
         s^{2\beta-\alpha} \left( h+1 \right)^{\alpha}
         +
         s^{2\beta-\alpha} h^{\alpha}
         \right)
         \\ & \qquad \qquad+
         \left( s+1 \right)^{2\beta} \psi \left( \frac{s+h+1}{s+1} \right)
         -
         \left( s+h \right)^{2\beta} \psi \left( \frac{s+1}{s+h} \right) \\
         &\qquad \qquad 
         -
         s^{2\beta} \psi \left( \frac{s+h+1}{s} \right)
         +
              s^{2\beta} \psi \left( \frac{s+h}{s} \right)
    \\ & \qquad =
          s^{2\beta-\alpha} \left( 2\lambda a_\alpha (h) + u_2(s,s+ h) \right),
  \end{align*}
where
\begin{align*}
  u_2(s,s+h) &= \notag
             \left( 1 - \left( 1 + \frac{1}{s} \right)^{2\beta-\alpha} \right) h^{\alpha}
             +
             \left(
             \left( 1 + \frac{h}{s} \right)^{2\beta-\alpha}
             -
             1
             \right)
             \left( 1-h \right)^{\alpha}
  \\ & \qquad + \notag
       s^{\alpha}
       \left(
       \left( 1 + \frac{1}{s} \right)^{2\beta}
       \psi \left( 1+ \frac{h}{s+1} \right)
       -
       \left( 1 + \frac{h}{s} \right)^{2\beta}
       \psi \left( 1 + \frac{1-h}{s+h} \right)
       \right.
  \\ & \qquad \left. - \notag
       \psi \left( 1 + \frac{h+1}{s} \right)
       +
       \psi \left( 1 + \frac{h}{s} \right) \right)
  \\ &  = 
             \left( 1 - \left( 1 + \frac{1}{s} \right)^{2\beta-\alpha} \right) h^{\alpha}
             +
             \left(
             \left( 1 + \frac{h}{s} \right)^{2\beta-\alpha}
             -
             1
             \right)
       \left( 1-h \right)^{\alpha}
       +
       s^{\alpha} v(s,h).
\end{align*} 
For the first part on the right-hand side, we have, using  the Mean Value Theorem,
\begin{align*}
&  \abs{
                 \left( 1 - \left( 1 + \frac{1}{s} \right)^{2\beta-\alpha} \right) h^{\alpha}
             +
             \left(
             \left( 1 + \frac{h}{s} \right)^{2\beta-\alpha}
             -
             1
             \right)
       \left( 1-h \right)^{\alpha}
     }\\
     & \qquad \leq  C \left(  s^{-1} h^{\alpha}+ (h/s)^{-1} (1-h)^\alpha \right)  \le C s^{-1}
\end{align*}
and we obtain the desired inequality. 

For $v(s,h)$, we first treat the case $\alpha<1$. In this case, it follows
straightforwardly from the Mean Value Theorem that $\abs{v(s,h)} \leq C s^{-1}$, which yields
$ s^\alpha \abs{v(s,h)} \leq C s^{\alpha-1}$.  
In the case $\alpha \geq 1$, a Taylor expansion in $s^{-1}$ around $0$ yields that  
\begin{align*}
  v&(s,h)
  \\ &=
       \left( 1 + \frac{1}{s} \right)^{2\beta}
       \psi \left( 1+ \frac{h}{s+1} \right)
       -
       \left( 1 + \frac{h}{s} \right)^{2\beta}
       \psi \left( 1 + \frac{1-h}{s+h} \right)
   -
       \psi \left( 1 + \frac{h+1}{s} \right)
       +
         \psi \left( 1 + \frac{h}{s} \right)
  \\ &=
       \left( 1 + 2\beta \frac{1}{s} + O \left( \frac{1}{s^2} \right) \right)
       \left(
       \psi(1) + \psi'(1) \frac{h}{s+1} + O \left( \frac{1}{s^2} \right)
       \right)
       \\ & \qquad -
            \left(
            1 + 2\beta \frac{h}{s} + O \left( \frac{1}{s^2} \right)
            \right)
            \left(
            \psi(1)
            +
            \psi'(1) \frac{1-h}{s+h} + O \left( \frac{1}{s^2} \right)
            \right)
  \\ &\qquad -
        \psi(1) - \psi'(1) \frac{h+1}{s} - O \left( \frac{1}{s^2} \right)
\\ &\qquad +
     \psi(1) + \psi'(1) \frac{h}{s} + O \left(  \frac{1}{s^2} \right)
\\     &=
       \psi(1) 2\beta \frac{1}{s} + \psi'(1) \frac{h}{s}
       - \psi'(1) \frac{1-h}{s}
       - \psi(1) 2\beta \frac{h}{s}
\\ &\qquad       - \psi'(1) \frac{h+1}{s}
     + \psi'(1) \frac{h}{s}
     + O \left( \frac{1}{s^2} \right)
  \\ &=
       2 (1-h) \left( \beta \psi(1) - \psi'(1) \right) \frac{1}{s}
       +
       O \left( \frac{1}{s^2} \right)
  \\ &=
       O \left( \frac{1}{s^2} \right),
\end{align*}
where we have used that $\beta \psi(1)=\psi'(1)$ to derive the last equality. This
yields~\eqref{eq:24a}.

Let us now show the claim   (\ref{eq:24}).  In this case, we have
  \begin{align*}
  &  \Ex\big[\left( X(s+h+1)-X(s+h) \right)\left( X(s+1) - X(s) \right)\big] \\
    &\qquad =
      \left( s+1 \right)^{2\beta}
      \left( 
      \phi \left( \frac{s+h+1}{s+1} \right)
      -
      \phi \left( \frac{s+h}{s+1} \right)
       \right)
   \\ &\qquad \qquad    -
      s^{2\beta} \phi \left( \frac{s+h+1}{s} \right)
      +
      s^{2\beta} \phi \left( \frac{s+h}{s} \right)
    \\ & \qquad =
         - \lambda
         \left(
         \left( s+1 \right)^{2\beta-\alpha}
         \left( h^{\alpha}
         -
         \left( h-1 \right)^{\alpha}
         \right)
         -
         s^{2\beta-\alpha} \left( \left( h+1 \right)^{\alpha}
         -
         h^{\alpha}
         \right) \right)
         \\ &\qquad \qquad +
              \left( s+1 \right)^{2\beta}
              \left( 
              \psi \left( \frac{s+h+1}{s+1} \right)
         -
               \psi \left( \frac{s+h}{s+1} \right)
               \right)
         \\ &\qquad \qquad -
              s^{2\beta}
              \left( 
              \psi \left( \frac{s+h+1}{s} \right)
         -
              \psi \left( \frac{s+h}{s} \right)
               \right)
    \\ & \qquad =
          s^{2\beta-\alpha} \left(2 \lambda a_\alpha(h) + u_2(s,s+h) \right),
  \end{align*}
where
\begin{align*}
  u_2(s,s+h) &=
             \left( 1 - \left( 1 + \frac{1}{s} \right)^{2\beta-\alpha} \right)
             \left( 
             h^{\alpha}
             -
             \left( h-1 \right)^{\alpha}
             \right)
\\ & \qquad +
     s^{\alpha}
     \left( 
     \left( 1 + \frac{1}{s} \right)^{2\beta}
     \left( 
       \psi \left( 1+ \frac{h}{s+1} \right)
       -
       \psi \left( 1 + \frac{h-1}{s+1} \right)
       \right) \right.
  \\ & \qquad \left.
       -
       \psi \left( 1 + \frac{h+1}{s} \right)
       +
       \psi \left( 1 + \frac{h}{s} \right) \right)
  \\ &=
             \left( 1 - \left( 1 + \frac{1}{s} \right)^{2\beta-\alpha} \right)
             \left( 
             h^{\alpha}
             -
             \left( h-1 \right)^{\alpha}
       \right)
       +
       s^{\alpha} w(s,h).
\end{align*}
By the Mean Value Theorem, we have that
\begin{equation*}
  \abs{
               \left( 1 - \left( 1 + \frac{1}{s} \right)^{2\beta-\alpha} \right)
             \left( 
             h^{\alpha}
             -
             \left( h-1 \right)^{\alpha}
             \right)
           }
           \leq C s^{-1} h^{\alpha-1},
\end{equation*}
which gives the desired estimate.
Furthermore,
\begin{align*}
   w(s,h)  &=
                \left( 1 + \frac{1}{s} \right)^{2\beta}
     \left( 
       \psi \left( 1+ \frac{h}{s+1} \right)
       -
       \psi \left( 1 + \frac{h-1}{s+1} \right)
           \right)
              -
       \psi \left( 1 + \frac{h+1}{s} \right)
       +
       \psi \left( 1 + \frac{h}{s} \right)
  \\ &=
       \left( 1 + \frac{1}{s} \right)^{2\beta}
       \int_{\frac{h-1}{s+1}}^{\frac{h}{s+1}} \psi'(1+y) dy
       -
       \int_{\frac{h}{s}}^{\frac{h+1}{s}} \psi'(1+y) dy
  \\ &=
       \left(
       \left( 1 + \frac{1}{s} \right)^{2\beta} - 1
       \right)
       \int_{\frac{h-1}{s+1}}^{\frac{h}{s+1}} \psi'(1+y) dy
       +
       \int_{\frac{h-1}{s+1}}^{\frac{h}{s+1}} \psi'(1+y) dy
       -
            \int_{\frac{h}{s}}^{\frac{h+1}{s}} \psi'(1+y) dy
  \\ &=
       \left(
       \left( 1 + \frac{1}{s} \right)^{2\beta} - 1
       \right)
       \int_{\frac{h-1}{s+1}}^{\frac{h}{s+1}} \psi'(1+y) dy
       \\ & \qquad 
       +
       \int_{0}^{\frac{1}{s+1}} \psi'\left(1+\frac{h-1}{s+1}+y\right) dy
       -
            \int_{1}^{\frac{1}{s}} \psi'\left(1+ \frac{h}{s} + y \right) dy
  \\ &=
       \left(
       \left( 1 + \frac{1}{s} \right)^{2\beta} - 1
       \right)
       \int_{\frac{h-1}{s+1}}^{\frac{h}{s+1}} \psi'(1+y) dy
       \\ & \qquad 
       +
            \int_{0}^{\frac{1}{s+1}}
            \left( 
            \psi'\left(1+\frac{h-1}{s+1}+y\right)
            -
            \psi'\left(1+ \frac{h}{s} + y \right) 
             \right)
            dy
            \int_{\frac{1}{s+1}}^{\frac{1}{s}} \psi'\left(1+ \frac{h}{s} + y \right) dy.
\end{align*}
Therefore, using the bounds on the derivatives of $\psi$ given by (H.1) and the fact that $h\le M_1 s+M_2$, we get that
\begin{equation*}
  \abs{w(s,h)}
  \leq
    C \left( s^{-2} + s^{-1-\alpha} h^{\alpha-1} \right).
\end{equation*}
This finishes the proof.
\end{proof}

 If $X$ is fBm with Hurst parameter $H\in(0,1)$, we have that
  \begin{equation*}
         \Ex\big[
       \left( X(t+1) - X(t) \right)
       \left( X(s+1) - X(s) \right)
     \big]
     =
     a_{2H}(s-t).
  \end{equation*}
Therefore, heuristically speaking, Lemma~\ref{lem:1} expresses that a process
  $X$ satisfying (H.1) is a ``perturbed''  fBm with Hurst
  parameter $\beta=\alpha/2$. 

For later reference, let us record here that the function $a_{\alpha}$ defined
in~\eqref{eq:6} has the asymptotics
\begin{equation}
  \label{eq:56}
  a_{\alpha}(h) = \frac{1}{2} \alpha (\alpha-1) \abs{h}^{\alpha-2} + o \left( \abs{h}^{\alpha-2} \right)
\end{equation}
as $\abs{h}\to \infty$. In particular, if $\abs{h} > \eta$, there exists a constant
$C_{\eta}$ such that 
\begin{equation}
  \label{eq:61}
  \abs{a(h)} \leq C_{\eta} \abs{h}^{\alpha-2}.
\end{equation}

Hypothesis (H.2) implies the following bound for the covariance.

\begin{lemma}
  \label{lem:3}
   Let $s,t > 0$ such that $s\wedge t \ge \eta>0$ and $\abs{s-t} \geq  (c-1)(s \land t) + c $, where $c$ is the constant appearing in hypothesis  (H.2). Then, assuming 
   (H.2), there exists a constant $C_\eta>0$ (not depending on  $s$ or $t$), such
   that 
   \begin{equation}
     \label{eq:9}
     \abs{\Ex\big[\left( X(s+1)-X(s) \right) \left( X(t+1)-X(t) \right)\big]}
     \leq
 C_\eta
    \begin{cases}
       (s \land t)^{2\beta+\nu-2} \abs{s-t}^{-\nu} & \text{if $\alpha<1$} \\
       (s \land t)^{2\beta-\alpha} \abs{s-t}^{\alpha-2} &  \text{if $\alpha \ge 1$}
     \end{cases},
   \end{equation}
   and the exponent $\nu$ is defined in hypothesis (H.2).
    \end{lemma}

 \begin{proof}
   Without loss of generality, we assume that $s \geq t$ so that $\abs{t-s} \geq (c-1) s
   \land t + c $ translates into $s \geq c(t+1)$.
    As $s \geq t$, we have by self-similarity that
   \begin{align*}
  &   \Ex\big[\left( X(s+1)-X(s) \right) \left( X(t+1)-X(t) \right)\big]\\
      &\qquad =
       \left( t+1 \right)^{2\beta} \left(
       \phi \left( \frac{s+1}{t+1} \right) - \phi \left( \frac{s}{t+1} \right)
       \right)
       -
       t^{2\beta}
       \left(
       \phi \left( \frac{s+1}{t} \right) - \phi \left( \frac{s}{t} \right)
       \right)
     \\ & \qquad =
          \left(
          \left( t+1 \right)^{2\beta} - t^{2\beta}
          \right)
          \left(
          \phi \left( \frac{s+1}{t+1} \right)
          -
          \phi \left( \frac{s}{t+1} \right)
       \right)
     \\ &\qquad \qquad +
       t^{2\beta}
          \left(
          \phi \left(  \frac{s+1}{t+1} \right)
          -
          \phi \left( \frac{s}{t+1} \right)
          -
          \phi \left( \frac{s+1}{t} \right)
          +
          \phi \left( \frac{s}{t} \right)
       \right).         
   \end{align*}
As $s \geq c (t+1)$, we have that $s/(t+1)\geq c$ and therefore, by (H.2), for each
   $x \in \left[ \frac{s}{t+1}, \frac{s+1}{t+1} \right]$,
   \begin{equation}
     \label{eq:11}
     \abs{\phi'(x)} \leq C
     \begin{cases}
       t^{\nu} \left( s-t \right)^{-\nu} &\qquad \text{if $\alpha < 1$} \\
       t^{2-\alpha} \left( s-t \right)^{\alpha-2} &\qquad \text{if $\alpha \geq 1$}
     \end{cases}
\end{equation}
and, for each $x \in \left[ \frac{s}{t+1}, \frac{s+1}{t} \right]$,
\begin{equation}
  \label{eq:10}
  \abs{\phi ''(x)} \leq C
  \begin{cases}
    t^{\nu + 1} \left( s-t \right)^{-\nu-1} &\qquad \text{if $\alpha <1$} \\
    t^{3-\alpha} \left( s-t \right)^{\alpha-3} & \qquad \text{if $\alpha \geq 1$}
  \end{cases}.
\end{equation}
This yields the assertion, as by the Mean Value Theorem,
\begin{equation*}
  \phi \left(  \frac{s+1}{t+1} \right) - \phi \left( \frac{s}{t+1} \right)
  =
  \frac{1}{t+1} \phi'(x_1)
\end{equation*}
and
\begin{align*}
          \phi \left(  \frac{s+1}{t+1} \right)
          -
          \phi \left( \frac{s}{t+1} \right)
          &-
          \phi \left( \frac{s+1}{t} \right)
          +
          \phi \left( \frac{s}{t} \right)
\\   &=
    \frac{1}{t+1} \phi'(x_2) - \frac{1}{t} \phi'(x_3)
  \\ &=
       \frac{1}{t+1} \left( \phi'(x_2) - \phi'(x_3)  \right)
       +
       \left( \frac{1}{t+1} - \frac{1}{t} \right) \phi'(x_3)
  \\ &=
       \frac{1}{t+1} \left( x_2 - x_3 \right) \phi''(x_4)
       -
       \frac{1}{t^2+t} \phi'(x_3),
\end{align*}
where the $x_i$ are some appropriate values in the correct intervals
for~\eqref{eq:11} and~\eqref{eq:10} to hold.  
\end{proof}

We can now proceed to the proof of Theorem \ref{thm:5}. 
In this section  $\mathfrak{H}$ will denote the Hilbert space defined as the closure of the set of step functions with respect to the scalar product
$\langle {\bf 1}_{[0,t]},{\bf 1}_{[0,s]}\rangle_\mathfrak{H} = \E[X(s)X(t)]$, $s,t\geq 0$ and, as before, we can  consider that $X$ as an isonormal Gaussian process indexed by $\HH$ ,  and defined on a probability space $(\Omega, \mathcal{F} , \mathbb{P})$. We will assume that $\mathcal{F}$ is generated by $X$.

First we will  prove the convergence of the finite dimensional
distributions of $F_\e$, separately in the two cases $\alpha< 2-\frac1d$ and
$\alpha= 2-\frac1d$  and later we will show tightness.

\subsection{Convergence of  finite-dimensional distributions: the case $\alpha < 2 - \frac 1d$}
 Fix an integer $p \ge
2$, choose times $0<t_1<\dots<t_p < \infty$,
and consider  the random vector $G^\e=(F_\e(t_1),\ldots, F_\e(t_p))$,
where  $F_\e$ has been defined in  (\ref{eq:2a}). We will make use of the notation
\[
\xi_t= \norm{X(t+1)-X(t)}_{L^2(\Omega)}
\]
and
\begin{equation}  \label{Phi}
\Phi(s,t)= \Ex{ [Y_1(s)Y_1(t)] } = \xi_s^{-1} \xi_t^{-1} \Ex{ [\Delta_1 X(s) \Delta_1 X(t) ] }.
\end{equation}
The chaotic expansion of $F_\e(t)$ 
is given by
\begin{equation}
  \label{eq:3}
  F_\e(t)=\sum_{q=d}^\infty I_q(g^\e_{t,q}),
\end{equation}
where, for each $t> 0$,
$$
g^\e_{t,q}=c_q\,\sqrt{\e}\int_0^{ t/\e}\xi_{u}^{-q}\,\partial_{u}^{\otimes q}\,du,
$$
and
$\partial_{u}={\bf 1}_{[u,u+1]}$. We will denote by $F_{q,\e} (t)= I_q(g^\e_{t,q})$ the projection of   $F_\e(t)$  on the $q$th Wiener chaos.

  We are now going to check that assumptions (a), (b) (c) and (d) of 
 Theorem  ~\ref{thm:4}
are satisfied by the family of $p$-dimensional vectors $G^\e$.

\medskip
\noindent
\textit{Proof of condition (a)}. Lemma~\ref{lem:4}
implies that,  for every $q \geq d$ and   for every $i,j\in\{1,\ldots,p\}$,  $q! \langle g^\e_{t_i,q},g^\e_{t_j,q} \rangle _{\HH^{\otimes q}} \to \sigma^2_{\alpha,q} (t_i \wedge t_j) $ 
as $\e\to 0$, where $\sigma^2_{\alpha,q}$ is given by~\eqref{eq:23}.

\medskip
\noindent
\textit{Proof of condition (b)}. This is straightforward.

\medskip
\noindent
\textit{Proof of condition (c)}. We have to show that for $r=1,2,\dots,q-1$ and for all $T>0$,
\begin{equation}
  \label{eq:74}
  \lim_{\e \to  0} \norm{  g^\e_{T,q} \otimes_r g^\e_{T,q}}_{\mathfrak{H}^{\otimes 2(q-r)}}^2 = 0.
\end{equation}
Using the notation (\ref{Phi}), we see that
\begin{eqnarray*}
  g^\e_{T,q} \otimes_r g^\e_{i,q}
 & =&  c_q^2\e
  \int_0^{  T/\e} \int_0^{  T/\e}  \xi_s^{-q}  \xi_t^{-q} \left\langle \partial_s, \partial _t \right\rangle_{\HH}^r \partial_s^{\otimes(q-r)}
  \otimes  \partial_t^{\otimes(q-r)} dsdt\\
   &=&
   c_q^2\e
  \int_0^{  T/\e} \int_0^{  T/\e}  \xi_s^{-(q-r)}  \xi_t^{-(q-r)}
  \Phi^r(s,t)
   \partial_s^{\otimes(q-r)} \otimes
   \partial_t^{\otimes(q-r)} dsdt.
\end{eqnarray*}
Therefore,
\begin{equation}
  \label{eq:53}
  \norm{  g^\e_{T,q} \otimes_r g^\e_{T,q}}_{\mathfrak{H}^{\otimes 2(q-r)}}^2
\\  =
c_q^4 \e ^2
       \int_{(0,   T/\e)^4}
       \Phi^r (s,t) \Phi^r(l,m) \Phi^{q-r} (s,l)  \Phi^{q-r} (t,m) 
       dsdtdldm.
\end{equation}

We claim that
\begin{equation}
  \label{eq:54}
 \sup_{\e>0} \e
  \int_0^{ T/ \e} \int_0^{ T/ \e}
  |\Phi^q(s,t)|
  dsdt
  \le C,
\end{equation}
where $C$ is some constant not depending on $q$ or $\e$.
Taking into account that $  |\Phi(s,t)| \le 1$, it suffices to show that
\[
 \sup_{\e>  0}  \e  
  \int_0^{ T/ \e} \int_0^{T/ \e}
  |\Phi^d(s,t)|
  dsdt <\infty.
  \]
By Lemmas~\ref{lem:10} and~\ref{lem:11},   it suffices to show that
\[
 \sup_{\e>  0}    \e
  \int_0^{ T/ \e} \int_0^{ T/ \e}
|a^d_{\alpha}(s-t)|
  dsdt <\infty,
  \]
  which is an immediate consequence of (\ref{eq:61}) and the fact that $\alpha <2-\frac 1d$.

Let  us write the integration domain $(0, T/\e)^4$ in the form $\cup_{i=1}^4 D_i$, where
\begin{eqnarray*}
 D_{1} &=& \left\{ (s,t,l,m) \in (0,T/\e)^4 :  |s-t| \ge (c-1) (s\wedge t) +c \right\},\\
  D_{2} &=& \left\{ (s,t,l,m) \in (0,T/\e)^4 :  |l-m| \ge (c-1)( l\wedge m )+c \right\},\\
   D_{3} &=& \left\{ (s,t,l,m) \in (0,T/\e)^4 :  |s-l| \ge  (c-1) (s\wedge l )+c \right\},\\
    D_{4} &=& \left\{ (s,t,l,m) \in (0,T/\e)^4 :  |t-m| \ge (c-1)( t\wedge m) +c \right\}.
 \end{eqnarray*}
 We claim that the integral over any of the sets $D_i$ converges to zero.  
By  H\"older's inequality, we have for nonnegative functions
$f_1,f_2,f_3,f_4$ and real numbers $x_i \le y_i$ for $i=1,2,3,4$ that
\begin{align*}
& 
 \e^2 \int_{x_1}^{y_1} \int_{x_2}^{y_2} \int_{x_3}^{y_3} \int_{x_4}^{y_4}
  f_1(s,t)
  f_2(l,m)
   f_3(s,l)
  f_4(t,m)
  dsdtdldm
  \\
  &\quad  \leq
  \left( \e
        \int_{x_1}^{y_1}
    \int_{x_2}^{y_2}
    f_1(s,t)^{q/r}
  \right)^{r/q}
  \left(\e
    \int_{x_3}^{y_3}
    \int_{x_4}^{y_4}
    f_2(l,m)^{q/r}
  \right)^{r/q}
  \\
 & \qquad \times \left(\e
  \int_{x_1}^{y_1}
  \int_{x_3}^{y_3}
  f_3(s,t)^{q/(q-r)}
\right)^{(q-r)/q} 
\left(\e
  \int_{x_2}^{y_2}
  \int_{x_4}^{y_4}
  f_4(s,t)^{q/(q-r)}
\right)^{(q-r)/q}.
\end{align*}
The above inequality, together with  (\ref{eq:54}) and  Lemmas
  \ref{lem:10} and   \ref{lem:4}, implies that the integral over  $\cup_{i=1}^4 D_i$ converges to zero. 
It therefore suffices to consider the integral over  $\cap _{i=1}^4  D_i^c$. Using  the decompositions
\[
\Phi^r(s,t)=R_{\alpha, r} (s,t) +   a^r_\alpha(s,t) 
\]
and
\[
\Phi^{q-r} (s,t) =R_{\alpha, q-r} (s,t) +   a^{q-r}_\alpha(s,t)  
\]
provided by Lemma  \ref{lem:11}, and applying the above H\"older inequality and    Lemma  \ref{lem:11} with $M_1=c-1$ and $M_2=c$, yields
\begin{align*}
  \lim_{\e \to 0}
  \norm{g_{T,q}^\e \otimes_r g_{T,q}^\e}_{\mathfrak{H}^{\otimes 2(q-r)}}^2&= 
  \lim_{\e \to  0}
c_q^4 \e^2
       \int_{ \cap _{i=1}^4  D_i^c}  a^r_{\alpha}(t-s)
        a^r_{\alpha}(m-l) \\
        &\qquad  \times 
            a^{q-r}_{\alpha}(l-s) 
              a^{q-r}_{\alpha}(m-t)
                 dsdtdldm.
           \end{align*}
It thus suffices to show that
   \begin{equation*}   
\lim_{\e \to  0}
  \e^2
  \int_{ (0,\frac T \e)^4}
  \abs{
        a^r_{\alpha}(t-s)
       a^r_{\alpha}(m-l)
a^{q-r}_{\alpha}(l-s)
       a^{q-r}_{\alpha}(m-t)}
dsdtdldm=0.
\end{equation*}
  As $a_{\alpha}$ is the covariance function of  a fractional Brownian
motion with Hurst parameter $\alpha/2$, this follows from the results in
Breton-Nourdin \cite{BreNou} or Darses-Nourdin-Nualart \cite{DNN}. 

\medskip
\noindent
\textit{Proof of condition (d).} We have to show that, for each $T>0$,
\[
  \lim_{N\to\infty} \sup_{\e  >0}
  \sum_{q=N+1}^{\infty} q! \norm{g^\e_{T,q}}_{\mathfrak{H}^{\otimes q}}^2
  =
  \lim_{N \to \infty} \sup_{\e >0}
  \sum_{q=N+1}^{\infty}
  c_q^2 q!
  \e
  \int_0^{  T/ \e} \int_0^{  T/ \e}
  \Phi^q(s,t)
  dsdt =0.
\]
As by assumption $\sum_{q=d}^{\infty} c_q^2 q! = \norm{f}^2_{L^2(\R,\gamma)}<\infty$, this follows
from  (\ref{eq:54}).

\medskip
As conditions (a)-(d) are verified, it follows that  the  random vector 
$(G^\e(t_1),\dots,G^\e(t_p))$ converges in distribution, as $\e$ tends to zero, to $N_p(0,\Sigma)$,
where $\Sigma= (\sigma_{i,j})_{1\le i,j \le p}$ is the matrix given by
\[
\sigma_{i,j} = \sigma^2 (t_i \wedge  t_j).
\]
Here, $\sigma^2$ is given by
(\ref{sigma}) with  $\rho(h) = a_{\alpha}(h)$ defined in (\ref{eq:6}).
  This completes the proof.\qed

\subsection{Convergence of  finite-dimensional distributions: the case $\alpha = 2 - \frac 1d$}
Fix an integer $p \ge
2$, choose times $0<t_1<\dots<t_p < \infty$,
and consider  the random vector $G^\e=(F_\e(t_1)/ \sqrt{ |\log \e|},\ldots, F_\e(t_p) / \sqrt{ |\log \e|} )  $,
where  $F_\e$ has been defined in  (\ref{eq:2a}).   As before, we need to show  that assumptions (a), (b) (c) and (d) of 
 Theorem  ~\ref{thm:4}
are satisfied by the family of $p$-dimensional vectors $G^\e$.

Condition (a) follows from Lemma~\ref{lem:4}, with  $\sigma_{i, j,d} = \sigma^2_{1- 2/d} (t_i \wedge t_j)$ and
$\sigma_{i, j,q} =0$ for $q>d$. Condition (b) is obvious.  In order to show conditions (c) and (d), let us first remark that  (\ref{eq:54}) is replaced here by
\begin{equation}
  \label{eq:54a}
 \sup_{\e>0}  \frac \e {| \log \e|}
  \int_0^{T/ \e} \int_0^{ T/ \e}
  |\Phi^q(s,t)|
  dsdt
  \le C,
\end{equation}
which follows from   Lemmas~\ref{lem:10} and~\ref{lem:11},  and the fact that 
\[
     \frac \e {| \log \e|}
  \int_0^{T /\e} \int_0^{ T/ \e}
|a^d_{\alpha}(s-t)|
  dsdt  \le 
       \frac  {C\e} {| \log \e|}
  \int_0^{T/ \e} \int_0^{ T/ \e}
  |t-s| ^{-1}
  dsdt
   <\infty,
  \]
By the same arguments as in the case $\alpha <2-\frac 1d$, condition (c) reduces to 
show that for any $r=1, \dots, q-1$,
 \begin{equation*}   
\lim_{\e \to  0}
   \frac   {\e^2} { (  \log \e)^2} 
       \int_{ (0,T/ \e)^4}
        a^r_{\alpha}(t-s)
       a^r_{\alpha}(m-l)
a^{q-r}_{\alpha}(l-s)
       a^{q-r}_{\alpha}(m-t)
dsdtdldm=0,
\end{equation*} 
and again this follows from the analogous result for the fractional Brownian motion.
Finally, condition (d) is a consequence of  (\ref{eq:54a}). 
\qed
 
\subsection{Proof of tightness}

Suppose first that $\alpha <2-\frac 1d$. 
It suffices to  show that for any $0\le s<t  $ and $\e >0$, there exists a constant $C_{p}>0$ such that
\[
\| F_\e(t) - F_\e(s) \|_{L^p(\Omega)}  \le C_{p} |t-s|^{1/2}.
\]
To show this inequality we will follow the methodology developed in the proof of Theorem \ref{thm1}.
The starting point of the proof is  the following representation of $f(Y_1(u))$ as an iterated divergence:
\[
f(Y_1(u))= \delta^d \left( f_d(Y_1(u))   \xi_u^{-d}  \partial_u ^{\otimes d} \right).
\]
Then  using Meyer's inequalities, we obtain
\begin{align*}
\| F_\e(t)- F_\e(s) \| _{L^p(\Omega)} &
=\sqrt{\e}  \left\| \int_{s/\e} ^{t/\e} f(Y_1(u)) du \right \| _{L^p(\Omega)} \\
&= \sqrt{\e}  \left\| \int_{s/\e} ^{t/\e}  \delta ^d \left( f_d(Y_1(u))  \xi_u^{-d} \partial_u ^{\otimes d} \right)d u\right \| _{L^p(\Omega)} \\
&\le 
 c_p  \sum_{k=0} ^d \sqrt{\e} \left\| \int_{s/\e} ^{t/\e}  D^k \left( f_d(Y_1(u) ) \xi_u^{-d} \partial_u ^{\otimes d} \right)du \right \| _{L^p(\Omega; \HH^{\otimes k})} 
\\&=:
c_p \sum_{k=0} ^d R_k.
\end{align*}
Using Minkowski and H\"older inequalities, we can write for any $k=0,1, \dots, d$,
\begin{align*}
R_k &=
  \sqrt{\e}  \left\| \int_{ [s/\e,t/\e]^2} f^{(k)}_d(Y_1(u))    f^{(k)}_d(Y_1(v)) \left(  \frac { \langle  \partial_u, \partial_v  \rangle_{\HH} } { \xi_u \xi_v} \right)^{d+k}  du dv\right \| _{L^{p/2}(\Omega)} ^{1/2}\\
& \le
   \| f^{(k)}_d \|_{L^p(\R,\gamma)}  \left( \e \int_{s/\e} ^{t/\e}  \int_{s/\e} ^{t/\e}     | \Phi^{d+k}(u,v)| dudv   \right) ^{1/2},
\end{align*}
where    $\Phi(u,v)$ has been defined in (\ref{Phi}).
  From the assumptions ofTheorem \ref{thm1} and (\ref{equ1}) it follows that the quantity 
$ \| f^{(d)}_d \|_{L^p(\R,\gamma)}$ is finite. Then, it suffices to show that for all $ 0\le s \le t $
\begin{equation} \label{ecu1}
A_\e := \e \int_{s/\e} ^{t/\e}  \int_{s/\e} ^{t/\e}      |\Phi^{ d}  (u,v)| dudv \le   t-s.
\end{equation}
In order to show (\ref{ecu1}), notice first that on the region where $u\le M$, $v \le M$ or  $|u-v| \le M$, we obtain the bound  $C M(t-s) $. Therefore,  it suffices to consider the integral over the region
\[
\mathcal{D}_{\e,M}= \{ (u,v) \in [s/\e, t/\e]^2 \cap [M,\infty)^2: |u-v| \ge M\}.
\]
We denote the corresponding term by
\[
\widetilde{A}_\e := \e \int_ {\mathcal{D}_{\e,M}}      |\Phi^{ d}  (u,v)| dudv.
\]
We are going to use two different estimates for $| \Phi^{ d}  (u,v)|$. First, on
the region $\{ (u,v) \in \mathcal{D}_{\e,M}: |u-v| \le (c-1) (s\wedge t) +c\}$, using
Lemmas  \ref{lem:2} and \ref{lem:1}, we have for large $M$,
\begin{eqnarray*}
  \Phi^{d}(u,v) 
 & =& \frac{1}{2^{d}(1+u_1(u))^{d}(1+u_1(u+h))^{d}} \\
 && \times 
    \left( \frac{u\wedge v }{u\vee v} \right)^{(\beta-  \frac  \alpha 2  )d}
    \left(
  2   a_\alpha(|u-v|) + u_2(u,v)
    \right)^{d} \\
    &\le &
    C |u-v| ^{ (\alpha -2) d}.
\end{eqnarray*}
Secondly,  on the region $\{ (u,v) \in \mathcal{D}_{\e,M}: |u-v| \ge (c-1) (s\wedge t)
+c\}$, using this time Lemmas \ref{lem:2} and \ref{lem:3}, we can write, again
for large $M$,
\begin{eqnarray*}
 | \Phi^{d} (u,v) |
 & \le & \frac{C_M} {2^{d}(1+u_1(u))^{d}(1+u_1(u+h))^{d}} \\
 &&\times
  \left(  (u\wedge v)^{ (\alpha + \nu -2)d } |u-v| ^{ -d\nu}   \mathbf{1}_{\alpha <1} +
  |u-v| ^{ d(\alpha -2)}   \mathbf{1}_{\alpha \ge 1} \right)
     \\
    &\le &
    C |u-v| ^{ (\alpha -2) d}.
\end{eqnarray*}
These estimates and the change of variable $(u,v) \to (u,u+h)$ lead to
\[
\widetilde{A}_\e  \le C (t-s) \int_M^\infty  h^{ (\alpha -2) d} dh.
\]
Under the condition  $\alpha < 2-\frac 1d$, the integral   $ \int_M^\infty  h^{ (\alpha -2) 2d} dh$ is finite and we obtain the desired estimate. 

Suppose now that $\alpha =2-\frac 1d$.  We claim that for any $0\le s<t  $ and $\e  \in (0,1)$, there exists a constant $C_{p}>0$ such that
\[
\frac 1 { \sqrt{ | \log \e |}} \| F_\e(t) - F_\e(s) \|_{L^p(\Omega)}  \le C_{p} |t-s|^{1/2}.
\]
The proof is analogous to the case  $\alpha < 2-\frac 1d$, and can be completed using the estimate
\[
\sup_{\e \in (0,1)} \frac 1 { \sqrt{ | \log \e |}}  \left( \int _M ^ {t/\e}   h^{ -1} dh \right) ^{1/2}   <\infty.
\]
 \qed

\section{Proof of Theorem~\ref{thm1.3}}
\label{sec5}

Recall the definition of  $\widetilde{F}_\e$   given by~\eqref{eq:2}. Denote
$\widehat{F}_\e=\e^{ 1/2-d(1-\alpha/2)}  \widetilde{F}_\e$ and let  $\widehat{F}_{q,\e}$  be the projection of
$\widehat{F}_\e$ on the $q$th Wiener chaos.
 Note that by assumption, the exponent $1/2- d (1-\alpha/2) $ is positive. 

\subsection{Convergence of finite-dimensional distributions}
We will show for $s,t \ge 0$ that
  \begin{align}
    \label{eq:17}
\lim_{\e,\delta \to  0}
\Ex{[\widehat{F}_{q,\e}(s) \widehat{F}_{q,\delta}(t))]}
&=
\begin{cases}
 c_d^2 K_d(s,t)
  & \qquad \text{if $q=d$,}
  \\
  0 & \qquad \text{if $q>d$,}
\end{cases}
 \end{align}
where $K_d(s,t)$ has been defined in (\ref{eq:70a}),
and also that
\begin{equation}
    \label{eq:44}
    \lim_{N \to \infty}
    \limsup_{\e \to  0}
    \sum_{q=N}^{\infty}       \Ex{[
      \widehat{F}_{q,\e}(t)^2]
    }
    =
    0.
  \end{equation}
  Then (\ref{eq:17}) for $q=d$ implies that for every sequence $\e_n \rightarrow 0$,  and for each $t\ge 0$, the sequence of random variables
  $\widehat{F}_{q,\e_n} (t)$ is  a Cauchy sequence
in $L^2(\Omega)$. Therefore,  $\widehat{F}_{d,\e} (t)$ 
  converges in $ L^2(\Omega)$  as $\e$ tends to zero  to $c_d H_{\infty}(t)$, where
  $H_\infty(t)$  is the generalized Hermite process with covariance given by
  (\ref{eq:70a}). Also, for $q > d$, $\widehat{F}_{q,\e}$ converges to zero in
  $L^2(\Omega)$, as $\e$ tends to zero.
  Together with~\eqref{eq:44}, this implies that for any $t\ge 0$,   $\widehat{F}_\e(t)$ converges in $L^2(\Omega)$, as $\e$ tends to zero,
 to  $c_d H_\infty(t)$. As a consequence, the finite-dimensional distributions of the process $\widehat{F}_\e$ converge in  law to those of the process $c_dH_\infty$.   This is also true for the finite-dimensional distributions of the process $\e^{1/2-d(1-\alpha/2)} F_\e$, because this process has the same law as  $\widehat{F}_\e$.
 
We now proceed to the proof of (\ref{eq:17}) and (\ref{eq:44}).
Taking into account that (see~\eqref{eq:3})
\[
\widehat{F}_{q,\e} (t)= c_q \e^{-d(1-\frac \alpha 2)} \int_0^ t  \| \Delta _\e X(u) \| _{L^2(\Omega)}  ^{-q} I_q( \mathbf{ 1}_{[u+\e, u]} ^{\otimes q}) du,
\]
we can write
\[
  \Ex{[
   \widehat{ F}_{q,\e}(s) \widehat{F}_{q,\delta}(t)]
  }
=
  c_q^2 q!( \e\delta)^{-d( 1-\frac \alpha 2)} 
  \int_0^{s} \int_0^{t}
 \Phi^q_{\e, \delta} (u,v) dudv,
\]
where
\begin{equation}  \label{phi2}
      \Phi_{\e,\delta}(u,v) =
    \frac{
      \Ex{ \left[
        ( X(u+ \e) - X(u))
        ( X(v+ \delta) - X ) (v)\right]
      }
    }{
      \sqrt{
        \Ex{
           \left[( X(u+ \e) - X(u))^2 \right]
        }
        \Ex{
          \left[ ( X(v+\delta) - X(v) )^2 \right]
        }
      }
    }.
\end{equation}
Assuming $t\ge \e$, consider the decomposition
\begin{eqnarray*}
 \Ex{[\widehat{ F}_{q,\e}(t) ^2]}& =&c_q^2 q! \e^{d(\alpha -2)} 
  \int_0^t
     \int_0^t
      \Phi^q_{\e ,\e}(u,v)
       du dv \\
       &= &
       c_q^2 q! \e^{d(\alpha -2)} 
  \int_{ [0,t]^2 \cap \{ u\wedge v \le \e\}}
      \Phi^q_{\e ,\e}(u,v)
       du dv \\
 &&       +
       c_q^2 q! \e^{ d(\alpha -2)} 
  \int_\e^t \int_ \e ^t     
      \Phi^q_{\e ,\e}(u,v)
       du dv.
       \end{eqnarray*}
Then, using the bound~\eqref{eq:50} and the fact that $| \Phi_{\e,\e} (u,v) | \le 1$, we can write for any $q\ge q^* \ge d$
such that $\alpha > 2 -\frac 1 {q^*}$
\begin{equation} \label{equ2}
 \Ex{[\widehat{ F}_{q,\e}(t) ^2]} \le  c_q^2 q!  t\e^{1+d(\alpha -2)} 
      +
       C c_q^2 q! \e^{(q^*-d)(2-\alpha )} 
  \int_\e^t \int_ \e ^t     \left(
 (u\vee v)^{\alpha -2} (uv)^{-\beta+\frac \alpha 2}
       \right)^{q^*}
       du dv.
\end{equation}
Notice that  the integral appearing in the right-hand side of the above equation is finite because $(-\beta + \frac \alpha 2) q^*\ge
( \frac \alpha 2-1) q^*>  -\frac 12$ and $(\alpha -2) q^* >-1$. As a consequence, if $ q>d$ we can choose $q^* >d$ with  $\alpha > 2 -\frac 1 {q^*}$ and we obtain
\[
\lim _{\e \rightarrow 0} \Ex{[\widehat{ F}_{q,\e}(t) ^2]}  =0.
\]
Furthermore, the estimate  (\ref{equ2}) also implies
\begin{equation}
\sup_{\e \in (0,1)} \sup_{q\ge d}    \Ex{[
   \widehat{ F}_{q,\e}(t) ^2]}  \le Cc_q^2 q!,
\end{equation}
which  allows us to establish  (\ref{eq:44}) because  the series $\sum_{q=d}^{\infty} c_q^2 q!$  is convergent.

It remains to show (\ref{eq:17}) for $q=d$. For any $\e >0$, we define
\[
\widehat{F}^{(1)} _{d,\e} (t)= c_d \e^{-d(1-\frac \alpha 2)} \int_0 ^\e  \| \Delta _\e X(u) \| _{L^2(\Omega)}  ^{-d} I_d( \mathbf{ 1}_{[u+\e, u]} ^{\otimes d}) du
\]
and  $\widehat{F}^{(2)} _{d,\e} (t)=\widehat{F} _{d,\e} (t)-\widehat{F}^{(1)} _{d,\e} (t)$.
It is easy to show that
\[
\lim _{\e \rightarrow 0} \Ex{[\widehat{ F}^{(1)}_{d,\e}(t) ^2]}  =0.
\]
Therefore, 
\[
\lim_{\e,\delta \rightarrow 0} 
\Ex{[
    \widehat{F}_{d,\e}(s) \widehat{F}_{d,\delta}(t)]}
    =
    \lim_{\e,\delta \rightarrow 0} 
\Ex{[
    \widehat{F}^{(2)}_{d,\e}(s) \widehat{F}^{(2)}_{d,\delta}(t)]}.
\]
We can write
\begin{equation}
  \label{eq:52}
  \Ex{[
    \widehat{F}^{(2)}_{d,\e}(s) \widehat{F}^{(2)}_{d,\delta}(t)]
  }
\\  =
    c_d^2 d!  
   \int_\e^s
     \int_\delta^t
     \left(
       (\e\delta)^{\alpha/2-1}
      \Phi_{\e ,\delta}(u,v)
       \right)^d
       du dv.
\end{equation}
By Lemma~\ref{lem:8}, we have that
\begin{align*}
  \lim_{\e, \delta \to 0}
     \int_\e^s
     \int_\delta^t
     &
     \left(
       (\e \delta)^{\alpha/2-1}
      \Phi_{\e,\delta}(u,v)
       \right)^d
       du dv
  \\ &=
     \int_0^s
     \int_0^t
       \left(
         \lim_{\e, \delta \to 0}
       (\e\delta)^{\alpha/2-1}
      \Phi_{\e,\delta}(u,v)
       \right)^d
       du dv
     \\ &=
            \int_0^s
            \int_0^t
            \left(
              \frac{
                \partial_{u,v} \Ex{[X (u) X(v)]}
              }{
                2\lambda \left( uv \right)^{\beta-\alpha/2}
              }
            \right)^d
           du dv,
\end{align*}
where the interchange of integration and limit is justified by the
bound~\eqref{eq:50}, which yields an integrable  bound since we have $d(\alpha-2)
>   -1$.   
 This completes the proof of  (\ref{eq:17}) and (\ref{eq:44}).
 
 \subsection{Tightness}
 To show tightness it suffices to estimate the  moment of order two of an increment. We can write, for $s\le t$,
 \begin{eqnarray*}
 \Ex{[ |
    \widehat{F}_{\e}(s) - \widehat{F}_{\e}(t) |^2]} &=&
     \e^{ d(\alpha -2)} \sum_{q=d } ^\infty c_q^2 q!    \int_s^t \int_s^t \Phi ^q_{\e,\e} (u,v) dudv \\
     & \le  &   \left(\sum_{q=d } ^\infty c_q^2 q!    \right)    \e^{ d(\alpha -2)}   \int_s^t \int_s^t  |\Phi ^d_{\e,\e} (u,v)| dudv.
\end{eqnarray*}
If   we integrate on the set where at least one of the variables is less than $\e$,  using H\"older's inequality  we obtain the bound
\[
C  \e^{ d(\alpha -2)}   (t-s)  \int_\R \mathbf{1} _{[s,t]} (u) \mathbf{1} _{[0,\e]} (u) du
\le C  \e^{ d(\alpha -2)+ \frac 1 {p_1} } (t-s) ^{1 + \frac 1{p_2}},
\]
where $1/ p_1 + 1/ p_2=1$. Choosing $p_1= 1/ ( d(2-\alpha)) >1$ we obtain a bound of the form  $C(t-s) ^{2-d(2-\alpha)}$. 
Notice that  $2-d(2-\alpha) >1$.

 On the other hand, if both variables are larger than $\e$, we can use the estimate \eqref{eq:50} and we obtain the bound
 \[
C  \int_s^t \int_s^t  ((u\vee v) ^{\alpha -2} (uv)^{ -\beta +\frac \alpha 2} )^d dudv.
  \]
  Making the change of variables $(u,v) \rightarrow (s+ x (t-s), s+ y(t-s)$, the above integral  can be bounded by
  \[
  C (t-s)^{2+ d( 2\alpha -\beta -2)}  \int_0^1 \int_0^1  ((x\vee y) ^{\alpha -2} (xy)^{ -\beta +\frac \alpha 2} )^d dxdy
  \]
  and again we obtain the desired estimate because  
  \[
  2+ d( 2\alpha -\beta -2)  \ge 2+ d( 2\alpha -3)  > d \ge 1.
  \]

\section{Technical lemmas}
\label{sec6}

\subsection{Lemmas for the case $\alpha \leq 2 - 1/q$}
\label{s-10}
Define
\begin{equation}  \label{help}
\varphi_{\alpha, q} (\e)=
 \begin{cases}
\e &  {\rm if} \ \alpha < 2- 1/q \\
\e /|\log \e |    & {\rm if}  \ \alpha =2- 1/q
\end{cases}.
\end{equation}
\begin{lemma}
  \label{lem:10}
  Assume (H.1) and (H.2).
  If $\alpha \leq 2 - 1/q$, then for any $T>0$,
  \begin{equation*}
    \lim_{\e \to   0}
  \varphi_{\alpha,q}(\e)
    \int_{0}^{T/\e } \int_{0}^{T/\e} \mathbf{1}_{D_2}(s,t) \abs{\Phi^q(s,t)} ds dt = 0,
  \end{equation*}
  where the set $D_2$ contains all tuples $(s,t) \in \R_+^2$ such that $\abs{s-t} \ge
(c-1) ( s \land t) + c$. Here, $c$ is the constant appearing in (H.2).
\end{lemma}

\begin{proof}
By symmetry, it suffices to consider the integral over the set
$\{ t \le s\}$. In this case, if $(s,t)\in D_2$, then $s-t \ge (c-1)t+c$, that is, $s\ge  c(t+1)$.

  Note that for all $\varepsilon \in (0,1)$,
  \begin{equation*}
    \int_0^{T/\e } \int_0^{T/\e} \mathbf{1}_{D_{\delta}}(s,t) dsdt
    \leq C_{T} \e \delta,
  \end{equation*}
  where $D_{\delta}$ contains all tuples $(s,t)$ such that at least one of the
  variables $s$ and $t$ is bounded by $\delta$. Together with the fact that
  $\abs{\Phi(s,t)}\leq 1$ we get that
  \begin{equation*}
   \varphi_{\alpha,q}(\e)
    \int_0^{T/\e } \int_0^{T/\e}
     \mathbf{1}_{D_{2,\delta}}(s,t) \abs{\Phi^q(s,t)} dsdt
    \leq C_{T} \delta
  \end{equation*}
  for all $\delta > 0$.
  It therefore suffices to show that
  \begin{equation*}
    \lim_{\e \to  0}
    \varphi_{\alpha,q}(\e)
    \int_{0}^{T/\e} \int_{0}^{T/\e} \mathbf{1}_{D_{2,\delta }}(s,t) \abs{\Phi(s,t)}^q dsdt = 0,
  \end{equation*}
  where $D_{2,\delta} = D_2 \cap D_{\delta}^c \cap \{ t\le s\}$ contains all tuples $(s,t)$ which are elements
  of $D_2$ and such that $s,t \geq \delta $.
  
  Then, we can apply  Lemma  \ref{lem:3}  and obtain that
  \begin{equation*}
   \abs{\Ex\big[\left( X(s+1)-X(s) \right) \left( X(t+1)-X(t) \right)\big]} \leq
C_\delta
    \begin{cases}
      t^{2\beta-2+\nu} \left( s-t \right)^{-\nu} & \qquad \qquad \text{if $\alpha < 1$} \\
      t^{2\beta-\alpha}(s-t)^{\alpha-2} & \qquad \qquad \text{if $\alpha \geq 1$}
    \end{cases},
 \end{equation*}
   where $C_\delta $ is a positive constant  depending on $\delta$.
   
 Moreover, we claim that
 \[
 \inf_{\delta \le s <\infty} ( 1+ u_1(s)) =b_\delta >0.
 \]
 In fact, for any $s\ge 0$ we have $\Ex\big[(X(s+1) - X(s))^2\big] >0$ (this is a consequence of the self-similarity property) and the map
 $t\rightarrow  \Ex\big[(X(s+1) - X(s))^2\big] $ is continuous. This implies that $1+u_1(s)$ is a strictly positive continuous function on $[\delta, \infty)$, which   is bounded by $1+K_\delta s^{-\delta _1}$, with $\delta_1>0$, for $s$ large enough, and for some constant $K_\delta$,  by Lemma   \ref{lem:2}. Notice that $u_1(s)$ may blow up at zero if $\alpha < 2\beta$. Therefore, we obtain
  \begin{equation*}
    | \Phi(s,t) | \leq
C_\delta b_\delta^{-1}
    \begin{cases}
      t^{\alpha-2+\nu} \left( s-t \right)^{-\nu} & \qquad \qquad \text{if $\alpha < 1$} \\
       (s-t)^{\alpha-2} & \qquad \qquad \text{if $\alpha \geq 1$}
    \end{cases}.
 \end{equation*}

  If $\alpha < 2-1/q$, we then
  get that 
  \begin{equation*}
    \int_0^{T/\e} \int_{0}^{T/\e} \mathbf{1}_{D_{2,\delta}}(s,t) \abs{\Phi^q(s,t)} dsdt
    \le
    C_\delta  b_\delta^{-1}
    \e^{-(\alpha-2)q-2}
  \end{equation*}
and the assertion follows as 
  \begin{equation*}
    \lim_{\e \to 0}  \e^{-(\alpha-2)q-2} \varphi_{\alpha,q}(\e) = 0.
  \end{equation*}
  In the case $\alpha=2-1/q \geq 1$, we obtain
  \begin{equation*}
    \int_0^{T/\e} \int_{0}^{T/\e} \mathbf{1}_{D_{2,\delta}}(s,t) \abs{\Phi^q(s,t)} dsdt
    \leq
     C_\delta  b_\delta^{-1} \e^{-1},
  \end{equation*}
  and again
  \begin{equation*}
    \lim_{\e \to  0} \e^{-1} \varphi_{\alpha,q}(\e) = 0.
  \end{equation*}
\end{proof}

\begin{lemma}
  \label{lem:11}
  Let $\alpha \leq 2 - 1/q$, assume (H.1) and (H.2) and, for $r=1,2,\dots,q$, define 
   \begin{equation}
     \label{eq:64}
     R_{\alpha,r}(s,t) = 
      \Phi(s,t)^r -  a_{\alpha}(s-t)^r.
  \end{equation}
  Then, for all $M_1,M_2>0$, it holds that
  \begin{equation}
    \label{eq:16}
    \lim_{ \e \to 0}
    \varphi_{\alpha,q}(\e)
    \int_{0}^{ T /\e} \int_{0}^{T /\e}
    \mathbf{1}_{D_{M_1,M_2}}(s,t)
    \abs{R_{\alpha, r}(s,t)}^{q/r}
    dsdt
    =
    0,
  \end{equation}
  where the set $D_{M_1,M_2}$ is given by
  \begin{equation}
    \label{eq:73}
    D_{M_1,M_2} = \left\{
      (s,t) \in \R^2 :
      \abs{s-t} \leq M_1 \left( s \land t \right) + M_2
    \right\}.
  \end{equation}
\end{lemma}

\begin{proof}
  From~\eqref{eq:64}, we see that
  \begin{equation}
    \label{eq:63}
    \sup_{(s,t) \in D_{M_1, M_2} } \abs{R_{\alpha, r}(s,t)} \leq C_{\alpha,q}.
  \end{equation}
  Indeed, by the  Cauchy-Schwarz inequality $\abs{\Phi(s,t)} \leq 1$    and, by~\eqref{eq:56}, 
$\abs{a_{\alpha}(s-t)} \leq C_{\alpha}$.
Also note that for $\delta >0$,
\begin{equation}
  \label{eq:71}
  \int_0^{T/\e} \int_0^{T/\e} \mathbf{1}_{D_{\delta}}(s,t) dsdt \leq C_{T} \, n \,  \delta,
\end{equation}
where $D_{\delta}$ consists of all tuples $(s,t) \in \R_+^2$ such that at least one of
the quantities $s,t,\abs{s-t}$ is less than $\delta$, and $C_{T}$ is some
positive constant.
The bounds~\eqref{eq:63} and~\eqref{eq:71} now yield that
\[
  \varphi_{\alpha,q}(\e)
    \int_{0}^{T/\e} \int_{0}^{T/\e}  \mathbf{1}_{D_\delta \cap D_{M_1,M_2}}(s,t)
    \abs{R_{\alpha, r}(s,t)}^{q/r}
    dsdt
  \leq C \delta
\]
for all $\delta > 0$.
Also, by symmetry it suffices to consider the integral over the set
$\{ (s,t)\in D_{M_1,M_2} : s<t\}$.
It therefore suffices to prove that
\begin{equation}
  \label{eq:20}
    \lim_{\e \to  0}
    \varphi_{\alpha,q}(\e)
    \int_{0}^{T/\e} \int_{0}^{T/\e}
    \mathbf{1}_{D_{M_1,M_2,\delta}}(s,t)
    \abs{R_{\alpha, r}(s,t)}^{q/r}
    dsdt
    =
    0,
  \end{equation}
  where 
  \[
  D_{M_1,M_2,\delta}=\{ (s,t): \delta \le  s<t ,  \delta \le t-s \le M_1 s+M_2\}.
  \] 
  For  $(s,t)\in   D_{M_1,M_2,\delta}$, Lemmas~\ref{lem:2} and~\ref{lem:1} apply and yield that
  \begin{align*}
    \Phi^r(s,t) &=
    \left( \frac  st
    \right)^{(\beta-\alpha/2)r}
    \frac{2^{-r}}{
      \left( 1+u_{1}(s) \right)^{r/2}
      \left( 1 + u_1(t) \right)^{r/2}
          }
    \left(
     2 a_{\alpha}(t-s) + u_2(s,t)
    \right)^r
    \\ &=
    \left(
      \frac st
    \right)^{(\beta-\alpha/2)r}
      \left( 1+u_{1}(s) \right)^{-r/2}
      \left( 1 + u_1(t) \right)^{-r/2}
         \\ & \qquad \qquad \times
    \left(
         a_{\alpha}(t-s)^r
         +
         \sum_{r'=1}^{r} \binom{r}{r'}  a_{\alpha}(t-s)^{r-r'}  2^{-r'} u_2(s,t)^{r'}
    \right).
  \end{align*}
Therefore, we have that
 \begin{equation*}
   R_{\alpha, r}(s,t)
   =
   \Phi^r -   a^r_{\alpha}(s-t) = \sum_{l=1}^3 R_{\alpha, r,l}(s,t),
 \end{equation*}
 where
 \begin{align*}
   R_{\alpha, r,1}(s,t) &=
                   \left( \frac s t  \right)^{(\beta-\alpha/2)r}
                      \left(     
      \left( 1+u_1(s) \right)^{-r/2} 
      \left( 1 + u_2(t) \right)^{-r/2} -1     \right)
\\ & \qquad 
     \times  \left ( a_{\alpha}(t-s) + 2^{-1} u_2(s,t)  \right)^r,
   \\
   R_{\alpha,r,2}(s,t)
                    &=
                      \left(
                            \left(\frac st  \right)^{(\beta-\alpha/2)r} - 1 \right)
                      a_{\alpha}(t-s)^r
  \\
   R_{\alpha,r,3}(s,t)
                    &=
                        \left(
      \frac st
    \right)^{(\beta-\alpha/2)r}
         \sum_{r'=1}^{r} \binom{r}{r'}  a_{\alpha}(t-s)^{r-r'}  2^{-r'} u_2(s,t)^{r'}.
 \end{align*}
 We will
now show that
\begin{equation}
  \label{eq:32}
    \lim_{\e \to  0}
    \varphi_{\alpha,q}(\e)
    \int_{0}^{T/\e} \int_{0}^{t}
    \mathbf{1}_{D_{M_1,M_2,\delta}}(s,t)
    \abs{R_{\alpha,r,l}(s,t)}^{q/r}
   dsdt  
    =
    0
  \end{equation}
  for $l=1,2,3$.
This then implies~\eqref{eq:20}
as
\begin{equation*}
  \abs{R_{\alpha,r}(s,t)}^{q/r} \le 3 \max_{l=1,2,3} \abs{R_{\alpha,r,l}(s,t)}^{q/r}
  \le 3 \sum_{l=1}^3 \abs{R_{\alpha,r,l}(s,t)}^{q/r}.
\end{equation*}

If not otherwise specified, all formulas proved throughout the rest of this
section are only claimed to be valid for $(s,t) \in D_{M_1, M_2,\delta}$. Furthermore, $C$
in the following denotes a generic positive constant which may change from line
to line and might depend on $\delta$. Dependence on variables is indicated as parameters.

Let us begin by treating $R_{\alpha,r,1}$.  By Lemma    \ref{lem:2}, we know that for some positive constant $C_{\delta}$ only depending on $\delta$, it holds that
  \begin{equation}
    \label{eq:12}
  \abs{u_1(s)}< C_{\varepsilon} s^{-\delta_1}   \qquad \text{and} \qquad \abs{u_1(t)} < C_{\varepsilon} t^{-\delta_1} 
\end{equation}
with $\delta_1=1- \alpha$ if $\alpha<1$, $\delta_1=2-\alpha$ if $\alpha \ge 1$.
The bound~\eqref{eq:12} and 
the Mean Value Theorem imply that
\begin{equation}
  \label{eq:22}
     \abs{
                       \left( 1+u_{1}(s) \right)^{-r/2}
      \left( 1 + u_1(t) \right)^{-r/2}
    -1}
    \leq C
    \left( 
      \abs{u_1(s)} + \abs{u_1(t)}
    \right)
    \leq
         C
         s^{-\delta_1}.
    \end{equation}
    Furthermore, taking into account the bounds~\eqref{eq:61}, ~\eqref{eq:19} and~\eqref{eq:19a},  we can write
        \begin{align}  \notag
      \label{eq:65}
      \abs{
        a_{\alpha}(t-s) + 2^{-1} u_2(s,t)
      }^r
&\le C \Big(  (t-s)^{(\alpha -2)r}+ ( s^{(\alpha -2)r} + s^{-r} ( t-s)^{(\alpha -1)r }) \mathbf{1}_{\{ t-s \ge 1\}} \\
& \qquad
+ \left( s^{(\alpha-1)r } \mathbf{1}_{\{\alpha <1\}} +  s^{(\alpha-2)r } \mathbf{1}_{\{\alpha \ge 1\}} \right) \mathbf{1}_{\{ t-s < 1\}} \Big).
      \end{align}
        The bounds~\eqref{eq:22} and~\eqref{eq:65} thus yield
    \begin{align*}
       \abs{R_{\alpha,r,1}(s,t)}^{q/r} & \le C
       \Big(  (t-s)^{(\alpha -2)q- \frac q r  \delta_1} \\
       &\qquad + ( s^{(\alpha -2)q- \frac q r  \delta_1}  
       + s^{-q- \frac q r  \delta_1} ( t-s)^{(\alpha -1)q- \frac q r  \delta_1 } \mathbf{1}_{\{ t-s \ge 1\}} \\
& \qquad
+ \left( s^{(\alpha-1)q- \frac q r  \delta_1 } \mathbf{1}_{\{\alpha <1\}} +  s^{(\alpha-2)q- \frac q r  \delta_1 } \mathbf{1}_{\{\alpha \ge 1\}} \right) \mathbf{1}_{\{ t-s < 1\}} \Big)
      \end{align*}
        and therefore, after a straightforward calculation,
     \begin{equation*}
       \int_{0}^{T/\e} \int_{0}^{t}
       \mathbf{1}_{D_{M_1, M_2, \delta}}(s,t)
      \abs{
      R_{\alpha,r,1}(s,t)
      }^{q/r}
  dsdt    
       \le
       C \e^{-(\alpha-2)q + \frac{q}{r} \delta_1 - 2}.
    \end{equation*}
    As $(\alpha-2)q- \frac{q}{r} \delta_1 + 2 \le 1 - \delta_1$, with equality only if $\alpha=2-1/q$
    and $1-\delta_1 \in (0,1]$, this shows~\eqref{eq:32} for $l=1$.
    
    Let us turn to $R_{\alpha,r,2}$. 
    We can write
         \begin{align*}
     &  \int_{0}^{T/\e} \int_{0}^{t}
       \mathbf{1}_{D_{M_1, M_2, \delta}}(s,t)
      \abs{
      R_{\alpha,r,2}(s,t)
      }^{q/r}
  dsdt    \\ 
     & \qquad   \le C\int_{0}^{T/\e} \int_{0}^{t}
       \mathbf{1}_{D_{M_1, M_2, \delta}}(s,t)
      \left(
                        1-    \left(\frac st  \right)^{(\beta-\alpha/2)r}  \right)^{q/r}
                    (t-s)^{(\alpha -2)q}   dsdt.
       \end{align*}
Making the change of variables $s=x/\e $ and $t=y/\e$,  the  integral in the right-hand side of the above inequality can be written as 
\[
 \e^{-(\alpha -2)q -2}\int_0^{T} \int_0^y   \mathbf{1}_{\{\frac \delta  \e \le x<y, \frac \delta  \e \le y-x \le M_1 x +\frac {M_2} \delta \}}\left( 1- \left( \frac xy \right) ^{(\beta-\alpha/2)r}  \right)^{q/r} (y-x)^{(\alpha -2)q}  dxdy.
\]
We claim that
\[
\int_0^T \int_0^y    \left( 1- \left( \frac xy \right) ^{(\beta-\alpha/2)r}  \right)^{q/r} (y-x)^{(\alpha -2)q}  dxdy <\infty.
\]
Indeed, with the change of variables $x=zy$, we obtain
\begin{align*}
&    \int_0^{T} \int_0^y    \left( 1- \left( \frac xy \right) ^{(\beta-\alpha/2)r}  \right)^{q/r} (y-x)^{(\alpha -2)q}  dxdy \\
 &\qquad    =\int_0^{t_1} \int_0^1    \left( 1- z ^{(\beta-\alpha/2)r}  \right)^{q/r} (1-z)^{(\alpha -2)q}  y^{(\alpha-2)q+1}dxdy <\infty.
    \end{align*}
    Finally,  this shows~\eqref{eq:32} for $l=2$ as  $\e^{-(\alpha-1)q-2}/\varphi_{\alpha,q}(n)$ converges
to zero as $n \to \infty$.

It remains to study $R_{\alpha,q,3}$. In this case, using~\eqref{eq:61} and 
the bounds  ~\eqref{eq:19} and~\eqref{eq:19a} for $u_2$, we get that~
\begin{align*}
  \abs{a_{\alpha}(s-t)^{r-s} u_2(s,t)^s}^{q/r}
  &\leq C
  (t-s)^{(\alpha-2) q(r-s)/r}  \\
  &\qquad \times \Big(
   ( s^{(\alpha -2)qs/r} + s^{-qs/r} ( t-s)^{(\alpha -1)qs/r }) \mathbf{1}_{\{ t-s \ge 1\}} \\
& \qquad
+ \left( s^{(\alpha-1)qs/r } \mathbf{1}_{\{\alpha <1\}} +  s^{(\alpha-2)qs/r } \mathbf{1}_{\{\alpha \ge 1\}} \right) \mathbf{1}_{\{ t-s < 1\}}  \Big).
 \end{align*}
 Therefore,  $  \abs{R_{\alpha,r,3}(s,t)}^{q/r}$ is also bounded by  the above quantity,  after a tedious but straightforward calculation, leads to
\begin{equation*}
  \int_{0}^{T/\e} \int_0^t \mathbf{1}_{D_{M_1,M_2,\varepsilon}}(s,t) \abs{R_{\alpha,r,3}(s,t)}^{q/r} dsdt
   \leq
   C
   \left( 
    | \log  \e|
      \e^{(2-\alpha)qs/r -1}
    \lor 
    \e^{(2-\alpha)q-2}
     \right).
  \end{equation*}
Noting that $(\alpha-2)qr'/r+1 \leq 1-r'/r < 1$, we obtain~\eqref{eq:32} for $l=3$, finishing the proof.
\end{proof}

\begin{lemma}
  \label{lem:4}
  In the setting introduced above, let $s,t \ge 0 $ and  let $q$ be a positive integer. Assume (H.1) and (H.2). Then, for $\alpha < 2 - \frac{1}{q}$, it holds that
  \begin{equation}
    \label{eq:5}
    \lim_{\e \to 0} \Ex\big[F_{q,\e}(s)F_{q,\e}(t)\big] = \sigma^2_{\alpha,q} (s\wedge t),
  \end{equation}
  and, for $\alpha=2-\frac{1}{q}$, it holds that
  \begin{equation}
    \label{eq:29}
    \lim_{\e \to 0 } \frac{1}{|\log \e|}  \Ex\big[F_{q,\e}(s)F_{q,\e}(t)\big] = \sigma^2_{2-1/q}(s\wedge t),
  \end{equation}
  where $\sigma^2_{\alpha,q}$  is given by 
  \begin{equation}
  \label{eq:23}
  \sigma^2_{\alpha,q} = c_q^2 q! \int_{\R} a_\alpha ^q (h)dh.
  \end{equation}
  if $\alpha<2-1/q$, and by  \eqref{eq:69}, if $\alpha=2-1/q$.
 \end{lemma}

 \begin{proof}
   Recall the definition~\eqref{help} of the helper function $\varphi_{\alpha,q}$.  We have that
  \begin{align}
    \notag
     \Ex\big[F_{q,\e}(s)F_{q,\e} (t)\big]
     &=   c_q^2   \varphi_{\alpha,q}(\e) \int_{0}^{s/\e} \int_{0}^{t/\e}
       \Ex\big[H_q(Y_1(u)) H_q(Y_1(v))\big] dsdt
     \\ &= \label{eq:67}
         c_q^2  q! \varphi_{\alpha,q}(\e)
          \int_0^{s/\e} \int_0^{t/\e}
          \Phi^q(s,t)
          dsdt.
  \end{align}
  By Lemmas~\ref{lem:10} and~\ref{lem:11}, we get that
  \begin{equation*}
    \lim_{\e \to 0}            \varphi_{\alpha,q}(\e)
          \int_0^{s/\e} \int_0^{t/\e}
          \Phi^q(s,t)
          dsdt
          =
    \lim_{\e \to 0}            \varphi_{\alpha,q}(\e)
    \int_0^{s/\e} \int_0^{t/\e}
    \mathbf{1}_{D}(s,t)
           a_{\alpha}(s-t)^q
          dsdt,
\end{equation*}
   where $a_{\alpha}$ is  defined in~\eqref{eq:6} and
   \begin{equation*}
     D = \left\{ (s,t) \in \mathbb{R}^2 : \abs{s-t} \le (c-1) (s \land t ) + c \right\}.
   \end{equation*}
 Then  the  convergences  follow from  the results for the fBm.
\end{proof}

\subsection{Lemmas for the case $\alpha > 2 - 1/d$}
\label{s-11}

\begin{lemma}
  \label{lem:7}
  Assume $\alpha > 2- 1/d$ and (H.2). Then, one has for $s,t > 0$ that
  \begin{equation}
    \label{eq:41}
    \abs{\partial_{s,t} \Ex\big[X(s) X(t)\big]}
    \leq
       C
   (s \land t)^{2\beta-\alpha} (s \lor t)^{\alpha-2},
 \end{equation}
 where $C$ is a positive constant only depending on $\alpha$ and $\beta$. In particular,
 for $u,v > 0$ it holds that
 \begin{equation}
   \label{eq:42}
    \lim_{\e, \delta \to 0} 
    \frac 1{ \e \delta}
       \Ex\big[
       \left( X(s + \e) - X(s) \right)
       \left( X(t + \delta) - X(t) \right) 
     \big]
     =
     \partial_{s,t} \Ex\big[X(s) X(t)\big].
   \end{equation}
\end{lemma}

\begin{proof}
  For $0 < s \le t$ we have by self-similarity that
  \begin{equation*}
    \Ex\big[X(s) X(t)\big] = s^{2\beta} \phi \left( \frac{t}{s} \right)
  \end{equation*}
 with $\phi(x) = \Ex\big[X(1) X(x)\big]$. A routine calculation yields that
 \begin{equation*}
   \partial_{s,t} \Ex\big[X(s) X(t)\big] =
   \left( 2\beta-1 \right) s^{2\beta-2} \phi' \left( \frac{t}{s} \right)
   -
   s^{2\beta-3} t \phi'' \left( \frac{t}{s} \right).
 \end{equation*}
 Using (H.2) if $t/s \ge c$ and   the fact that $\phi'(t/s)$ is bounded and 
  $|\phi''(t/s) | \le C (t/s -1) ^{\alpha -2} $   if $t/s \le c$,  we obtain   
 \begin{equation}
   \label{eq:39}
   \abs{ \partial_{s,t} \Ex\big[X(s) X(t)\big]}
   \le
   C
   s^{2\beta-\alpha} t^{\alpha-2},
 \end{equation}
 which proves the asserted bound. As by assumption $\alpha-2 > -1/d \geq -1$ and by
 definition $2\beta-\alpha > 0$, the derivative
 is therefore integrable on any interval $[0,a] \times [0,b]$ and we get that
 \begin{equation*}
\Ex\big[
       \left( X(s + \e) - X(s) \right)
       \left( X(t + \delta) - X(t) \right) 
   \big]
   =
   \int_t^{t+\delta} \int_s^{s+\e} \partial_{u,v} \Ex\big[X(u) X(v)\big] du dv,
 \end{equation*}
 so that~\eqref{eq:42} follows.  
\end{proof}

\begin{lemma}
  \label{lem:8}
  Assume $\alpha > 2- 1/d$ and (H.2) and recall that $\Phi_{\e,\delta} (s,t)$ has been introduced in (\ref{phi2}).
  Then it holds that
  \begin{equation}
    \label{eq:48}
    \lim_{\e,\delta \to  0} \left( \e \delta \right)^{\alpha/2-1} \Phi_{\e,\delta}(s,t)
    =
             \frac{
         \partial_{s,t}
         \Ex\big[X(s) X(t)\big]
         }{
         2\lambda \left( st \right)^{\beta-\alpha/2}
         }.
\end{equation}
  Furthermore, there exists a positive constant $C_T$ such that for  any  $s,t \in [0,T]$ such that  $s \ge \e $ and $t\ge \delta$,
  it holds that
  \begin{equation}
    \label{eq:50}
    \abs{\Phi_{\e, \delta }(s,t)} \leq
    C_T     \left( \e\delta \right)^{1-  \frac \alpha 2 }
    \left( s \lor t  \right)^{\alpha-2} (st) ^{-\beta +\frac \alpha 2} .
  \end{equation}
\end{lemma}

\begin{proof}
  By self-similarity and Lemma \ref{lem:2}, we have that
  \begin{align*}
    \Ex\left[
      \left(
        X(s+\e) - X(s)
      \right)^2
    \right]
    &=
      \e^{2\beta}
         \Ex\big[
      \left(
        X(s/\e +1)- X(s / \e )
      \right)^2
      \big]
    \\ &=
         2\lambda \e^{\alpha} s^{2\beta-\alpha} \left( 1 + u_1(s/\e) \right),
  \end{align*}
  where $u_1(s/\e) \leq C \left( s/\e \right)^{\alpha-2}$ converges to zero as $\e \to 0$.
  Therefore, also using Lemma~\ref{lem:7}, we have that
  \begin{align*}
    \lim_{\e, \delta \to  0} \left( \e\delta \right)^{\alpha/2-1} \Phi_{\e,\delta}(s,t)
    &=
         \lim_{\e,\delta \to 0}
          ( \e\delta)^{-1}
         \frac{
                  \Ex\big[
         \left(
         X(s+\e) - X(s)
         \right)
         \left(
                  X(t+\delta) - X(t)
         \right)
         \big]
         }{
         2\lambda \left( st \right)^{\beta-\alpha/2}
         }
    \\ &=
         \frac{
         \partial_{s,t}
         \Ex\big[X(s) X(t)\big]
         }{
         2\lambda \left( st \right)^{\beta-\alpha/2}
         }.
\end{align*}
In order to establish the bound ~\eqref{eq:50},  using Lemma \ref{lem:2} and the condition $s/\e \ge1$ and $t/\delta \ge 1$,
  we obtain
\[
 | \Phi_{\e,\delta} (s,t)| \le   C  (\e\delta)^{- \frac \alpha  2}  (st) ^{- \beta + \frac  \alpha 2 } \int_t  ^{t+\delta}  \int_s ^{s+\e}  \partial_{u,v}  \Ex\big[X(u) X(v)\big] dudv.
\]
 Finally, in view of the estimate  ~\eqref{eq:41}, we can write
 \[
 | \Phi_{\e,\delta} (s,t)| \le   C  (\e\delta)^{1- \frac \alpha  2}  (st) ^{- \beta + \frac  \alpha 2 }   (s\vee t) ^{\alpha -2}.
 \]
 This completes the proof of the lemma.
\end{proof}

\subsection{Chaos expansion of the absolute value}
\label{s-12}
The next statement has been used in the end of the introductory section, when we
applied Theorems \ref{thm:5}  and \ref{thm1.3} in the case where
$X=\widetilde{B}$ is a bifractional Brownian motion, and when for $f$ we choose
the function $f(x)=\abs{x} - \sqrt{\frac{2}{\pi}}$. The proof is well-known and
standard, we include it here for completeness.

\begin{proposition}[Chaos expansion of the absolute value]
\label{prop:1}
 It holds that
\begin{equation*}
  \abs{x}
  =
  \sqrt{\frac{2}{\pi}}
       + \sum_{q=1}^{\infty} \frac{1}{q! \, (2q-1)} \, H_{2q}(x),
  \qquad x \in 
  \mathbb{R}.
\end{equation*}
\end{proposition}

\begin{proof}
The absolute mean of a standard Gaussian is
$\sqrt{\frac{2}{\pi}}$. By symmetry and the fact that
\begin{equation*}
  H_q(x) = (-1)^q \phi(x)^{-1} \phi^{(q)}(x),
\end{equation*}
with $\phi(x)=\frac{1}{\sqrt{2\pi}}e^{-x^2/2}$ the Gaussian density,
we get $ \int_{\R}^{} \abs{u} H_{2q+1}(u) \phi(u) du=0$ and
\begin{align*}
  \int_{\R}^{} \abs{u} H_{2q}(u) \phi(u) du
  &=
    2 \int_0^{\infty} u H_{2q}(u) \phi(u) du
    =
    2  \int_0^{\infty} u  \phi^{(2q)}(u) du
  \\ &=
       - 2 \int_0^{\infty}  \phi^{(2q-1)}(u) du
  =
       2 \phi^{(2q-2)}(0)
  =
       2 \cdot (2q-3)!!.
\end{align*}
Furthermore, we have that $\int_{-\infty}^{\infty} H_{2q}(u)^2 \phi(u) du=(2q)!$. Therefore,
\begin{align*}
  \abs{x}
  & =
       \sum_{q=0}^{\infty} \frac{\int_{-\infty}^{\infty} \abs{u} H_{2q}(u) \phi(u) du}{\int_{-\infty}^{\infty} H_{2q}(u)^2 \phi(u) du} H_{2q}(x)
 =
       \sqrt{\frac{2}{\pi}}
       +
       \sum_{q=1}^{\infty} \frac{1}{q!\, (2q-1)} H_{2q}(x).
\end{align*}
\end{proof}


\begin{thebibliography}{9}

\bibitem{BenHariz}
S. Ben Hariz (2002):
Limit theorems for the non-linear functionals of stationary Gaussian processes.
{\it J. Mult. Anal.} {\bf 80}, pp. 191-216.

\bibitem{BreNou}
J. C. Breton and I. Nourdin (2008):
Error bounds on the non-normal approximation of Hermite power variations of fractional Brownian motion.
{\it Electron. Commun. Probab.} {\bf 13}, paper no. 46, pp. 482-493.

\bibitem{BM} 
P. Breuer and P. Major (1983):
Central limit theorems for non-linear functionals of Gaussian fields.
{\it J. Mult. Anal.} {\bf 13}, pp. 425-441.

\bibitem{ChambersSlud}
  D. Chambers and D. Slud:
  Central limit theorems for nonlinear functionals of stationary Gaussian
  processes.
  {\it Probab. Theory Related Fields} {\bf 80}, pp. 323-346

\bibitem{DNN}
S. Darses, I. Nourdin and D. Nualart (2010):
Limit theorems for nonlinear functionals of Volterra processes via white noise analysis.
{\it Bernloulli} {\bf 16}, no. 4, pp. 1262-1293.

\bibitem{DM}
R. L. Dobrushin and P. Major (1979):
Non-central limit theorems for non-linear functions of Gaussian fields.
{\it Z. Wahrscheinlichkeitstheorie verw. Gebiete} {\bf 50}, pp. 27-52.

\bibitem{HN}
D. Harnett and D. Nualart (2018):
Central limit theorem for functionals of a generalized self-similar Gaussian process.
{\it Stoch. Proc. Appl.} {\bf 128}, no. 2, pp. 404-425.

\bibitem{HuNu}
Y. Hu and D. Nualart (2005):
Renormalized self-intersection local time for fractional Brownian motion
{\it Ann. Probab.} {\bf 33}, no. 3, pp. 948-983.

\bibitem{JN}
A. Jaramillo and D. Nualart (2018):
Functional limit theorem for the self-intersection local time of the fractional Brownian motion.
{\it Ann. Inst. H.  Poincar\'e}. To appear.

\bibitem{LN}
P. Lei and D. Nualart (2009):
A decomposition of the bifractional Brownian motion and some applications. 
{\it Stat. Probab. Lett.} {\bf 79}, no. 5, pp. 619-624.

\bibitem{NNT}
I. Nourdin, D. Nualart and C. A. Tudor (2010):
Central and non-central limit theorems for weighted power variations of fractional Brownian motion.
{\it Ann. Inst. H. Poincar\'e} {\bf 4}, no. 4, pp. 1055-1079.


\bibitem{IvanGioBook}
I. Nourdin and G. Peccati (2012):
{\it Normal approximations with Malliavin calculus: from Stein's method to universality}.
Cambridge tracts in Mathematics {\bf 192}, Cambridge University Press.

\bibitem{DavidBook}
D. Nualart (2006):
{\it The Malliavin calculus and related topics}.
2nd edition.
Probability and Its Applications, Springer.

\bibitem{DavidEulaliaBook}
D. Nualart and E. Nualart (2018):
{\it Introduction to Malliavin Calculus}.
 Institute of Mathematical Statistics Textbooks, Cambridge University Press.

\bibitem{NuPe}
D. Nualart and G. Peccati (2005):
Central limit theorems for sequences of multiple stochastic integrals.
{\it Ann. Probab.}  {\bf 33}, no. 1,  pp. 177-193.


\bibitem{PeTu}
G. Peccati and C. A. Tudor (2005):
Gaussian limits for vector-valued multiple stochastic integrals.
{\it Lecture Notes in Math.} {\bf 1857}, pp. 247-262.

\bibitem{T}
M.S. Taqqu (1979):
Convergence of integrated processes of arbitrary Hermite rank.
{\it Z. Wahrscheinlichkeitstheorie verw. Gebiete} {\bf 50}, pp. 53-83

\end{thebibliography}
\end{document}